\documentclass[11pt,a4paper]{amsart}

\usepackage[margin=2.6cm]{geometry}
\usepackage{amsmath,amssymb,amsthm}
\usepackage{mathtools}
\usepackage{booktabs}
\usepackage{enumitem}
\usepackage{xcolor}
\usepackage[colorlinks=true,linkcolor=blue!50!black,citecolor=blue!50!black,urlcolor=blue!50!black]{hyperref}

\theoremstyle{plain}
\newtheorem{theorem}{Theorem}[section]
\newtheorem{proposition}[theorem]{Proposition}
\newtheorem{lemma}[theorem]{Lemma}
\newtheorem{corollary}[theorem]{Corollary}

\newtheorem{thmx}{Theorem}

\theoremstyle{definition}
\newtheorem{definition}[theorem]{Definition}
\newtheorem{example}[theorem]{Example}
\theoremstyle{remark}
\newtheorem{remark}[theorem]{Remark}

\newcommand{\C}{\mathbb{C}}
\newcommand{\R}{\mathbb{R}}
\newcommand{\PP}{\mathbb{P}}
\newcommand{\Gr}{\mathrm{Gr}}
\newcommand{\Iso}{\mathrm{Iso}}
\newcommand{\SU}{\mathrm{SU}}
\newcommand{\SL}{\mathrm{SL}}
\newcommand{\GL}{\mathrm{GL}}
\newcommand{\U}{\mathrm{U}}
\newcommand{\Hom}{\mathrm{Hom}}
\newcommand{\sig}{\operatorname{sig}}
\newcommand{\rank}{\operatorname{rank}}
\newcommand{\codim}{\operatorname{codim}}
\newcommand{\Ep}{E_{+}}
\newcommand{\Em}{E_{-}}
\newcommand{\cO}{\mathcal{O}}
\newcommand{\cD}{\mathcal{D}}
\newcommand{\cA}{\mathcal{A}}
\newcommand{\cP}{\mathcal{P}}
\newcommand{\la}{\langle}
\newcommand{\ra}{\rangle}
\newcommand{\Om}{\Omega_{\mathrm{AG}}}

\begin{document}

\title[Holomorphic hulls of Matsuki orbits]{Holomorphic functions and hulls of Matsuki-dual orbits\\ of $\SU(p,q)$ in complex Grassmannians}

\author{Irfan Ullah}
\address{Abdus Salam School of Mathematical Sciences, Government College University, Lahore 54000, Pakistan}
\address{Department of Sciences and Humanities, FAST National University of Computer and Emerging Sciences, Islamabad 44000, Pakistan}
\email{irfanullah@nu.edu.pk}


\begin{abstract}
Let $G_0$ be a real form of a complex semisimple group acting on a flag manifold $Z$, with maximal compact subgroup $K_0$ and its complexification $K$. Matsuki duality pairs each $G_0$-orbit $\gamma$ with a $K$-orbit $\kappa$, and $M=\kappa\cap\gamma$ is a compact $K_0$-orbit. For lower-dimensional $\gamma$, $M$ is a compact CR manifold, and we ask which functions are holomorphic near $M$ and what its holomorphic hull is. In general, we show that $K_0$-finite holomorphic functions near $M$ are regular functions on $K$-orbits, and that $M$ carries non-constant holomorphic germs if some $K$-stable prime divisor of $Z$ misses $\kappa$, and, under a connectedness hypothesis, only constants otherwise. For $G_0=SU(p,q)$ acting on any Grassmannian $\mathrm{Gr}_k(\mathbb{C}^{p+q})$ we give a complete answer. Holomorphic functions near $M$ are constant unless $M$ is of holomorphic type: $k\in\{p,q\}$ and $M$ lies in the Harish-Chandra chart of the bounded symmetric domain $\mathcal{D}$. Then $M$ is the manifold of rank-$c$ partial isometries, and a holomorphic hull of $M$ is $\overline{\mathcal{D}}\cap\{\mathrm{rank}\le c\}$. This agrees with the polynomial hull of Kaup and Zaitsev except in one totally real case, where it is strictly smaller because $\det^{-1}$ is holomorphic near $U(p)$. The same dichotomy decides when global CR functions on $M$ are constant and when holomorphic sections of line bundles near $M$ are global; its exceptions differ from the Gindikin-Matsuki exceptions for cycle spaces, except at the base point of $\mathcal{D}$. This settles the descent question left open in earlier work with K. Shabbir on rank-one orbits. We also determine which $K$-orbits carry non-constant regular functions, prove constancy via chains of rational curves for small $c$, and compute Levi forms in examples.
\end{abstract}

\maketitle

\tableofcontents

%=====================================================================
\section{Introduction}\label{sec:intro}
%=====================================================================

Let $G$ be a connected complex semisimple Lie group, $G_0\subset G$ a non-compact real form, $K_0\subset G_0$ a maximal compact subgroup and $K=K_0^\C$. On a $G$-flag manifold $Z=G/P$ both $G_0$ and $K$ have finitely many orbits \cite{Wolf69}. Matsuki duality \cite{Matsuki79,Matsuki88,MUV92,BL02} is a bijection $\gamma\leftrightarrow\kappa$ between $G_0$-orbits and $K$-orbits characterised by the condition that
\[
M:=\kappa\cap\gamma
\]
is non-empty and compact; $M$ is then a single $K_0$-orbit. We call $M$ the \emph{Matsuki orbit} of $\gamma$.

When $\gamma=D$ is open (a \emph{flag domain}), $M=C_0$ is the base cycle: a compact complex submanifold of $D$. Its deformations form the cycle space of $D$, which is a Stein domain and, with explicit Hermitian exceptions, coincides with the universal Akhiezer--Gindikin domain $\Om$ \cite{FHW}. When $\gamma$ is lower-dimensional, Huckleberry and Ntatin \cite{HN03,Ntatin} defined and identified its cycle space. However, $M$ itself is then only a compact CR submanifold of $Z$, and its complex-analytic properties remained undescribed. This raises the following questions: what are these partially complex orbits, which functions are holomorphic near them, and what is their envelope (hull) of holomorphy? The guiding example is $\SU(p,1)$ acting on $\PP^p$. There the closed orbit is a sphere, every holomorphic function near it extends to the ball it bounds, and the ball contains the base cycle of the adjacent flag domain.

This paper answers these questions completely for $G_0=\SU(p,q)$ acting on an arbitrary complex Grassmannian $\Gr_k(\C^{p+q})$, and proves two general results along the way. The answer is a sharp dichotomy. Generically, a neighbourhood of $M$ carries only constants. The sphere--ball phenomenon survives precisely for the orbits that live in the Harish-Chandra realisation of the bounded symmetric domain $\cD$ of $\SU(p,q)$. These orbits are the manifolds of partial isometries (tripotents) studied by Kaup and Zaitsev \cite{KZ00,KZ03}.

\subsection*{Relation to fibrewise Rossi extension domains}
This paper continues the work of the author and K.~Shabbir \cite{UllahShabbir}, which studied the same problem for the rank-one Matsuki orbits $M(a,b,1)$ with $\min(p-a,q-b)=1$ and $\max(p-a,q-b)\ge2$. There each fibre of $M\to\Gr_a(\Ep)\times\Gr_b(\Em)$ is a strictly pseudoconvex sphere bounding a projective ball. It was shown in \cite[Thm.~1]{UllahShabbir} that every real-analytic CR function on $M$ extends holomorphically to the union of these balls, the \emph{fibrewise Rossi extension domain} $D_{a,b,1}$. This domain lies in an auxiliary projective bundle $X_{a,b}$, and the question whether the extensions descend to the Grassmannian $Z$ under the evaluation map $X_{a,b}\to Z$ was left open there. The present paper extends that work in three directions:
\begin{itemize}
\item it treats Matsuki orbits of arbitrary rank $c$ in every Grassmannian $\Gr_k(\C^{p+q})$, not only rank one;
\item it works directly in $Z$ and determines the global holomorphic hull of $M$ and the maximal ideal space of $CR^0(M)$, rather than a domain built fibre by fibre;
\item it answers the descent question of \cite{UllahShabbir} in the rank-one case. If $M$ is not of holomorphic type, all CR functions are constant. If $M$ is of holomorphic type, the evaluation map sends the closure of $D_{a,b,1}$ onto the hull $\widehat M$, and the fibrewise extensions descend to it.
\end{itemize}
The precise relation is worked out in Section~\ref{subsec:earlier}.

\subsection*{Main results}
Write $n=p+q$, $\C^n=\Ep\oplus\Em$ for an orthogonal splitting into a positive $p$-plane and a negative $q$-plane of the Hermitian form $h$ of signature $(p,q)$, and $Z=\Gr_k(\C^n)$. The $G_0$-orbits are $O(a,b,c)$, indexed by the signature $(a,b,c)$ of $h|_W$ ($a$ positive, $b$ negative, $c$ null directions). Their Matsuki orbits are denoted $M(a,b,c)$ (Section~\ref{sec:orbits}).

\begin{definition}\label{def:holtype}
$M(a,b,c)$ is of \emph{holomorphic type} if $b=0$ and $k=p$, or $a=0$ and $k=q$.
\end{definition}

If $b=0$ and $k=p$, then every $W\in M$ satisfies $W\cap\Em=0$, so $W$ is the graph of a linear map $\Phi:\Ep\to\Em$. The chart $\cA=\{W\cap\Em=0\}\cong\Hom(\Ep,\Em)$ contains the open orbit $O(p,0,0)$, which in this chart is the open matrix ball $\cD=\{\|\Phi\|<1\}$, the bounded symmetric domain of $\SU(p,q)$ in its Harish-Chandra realisation. Throughout, $\overline\cD=\{\|\Phi\|\le1\}$ denotes its closure, which is compact and lies in $\cA$; all hulls below are closed subsets of $\overline\cD$. The second case is symmetric.

\begin{thmx}[Structure]\label{thmA}
$M(a,b,c)$ is a $K_0$-homogeneous fibre bundle over $\Gr_a(\Ep)\times\Gr_b(\Em)$ whose fibre is the closed $\SU(p-a,q-b)$-orbit $\Iso_c(\C^{p-a,q-b})$ of totally isotropic $c$-planes. It is a generic CR submanifold of its $K$-orbit $\kappa(a,b)$, of CR codimension $c^2$. Its radical map $W\mapsto W\cap W^{\perp_h}$ exhibits it also as a bundle with compact complex fibres $\Gr_a(\C^{p-c})\times\Gr_b(\C^{q-c})$ over $\Iso_c(\C^n)$.
\end{thmx}

\begin{thmx}[Algebraisation; general $G_0$]\label{thmB}
Let $y\in Z$, let $U$ be a $K_0$-invariant neighbourhood of $K_0y$ in $Z$, and suppose the stabiliser $K_y$ is connected. Every $K_0$-finite holomorphic function on $U$ coincides, near $K_0y$ in $Ky$, with a unique regular function on the quasi-projective variety $Ky$.
\end{thmx}

\begin{thmx}[Divisor criterion; general $G_0$]\label{thmC}
Let $M=\kappa\cap\gamma$ be a Matsuki orbit in a $G$-flag manifold $Z$.
\begin{enumerate}[label=\textup{(\roman*)}]
\item If some $K$-stable prime divisor $D\subset Z$ does not contain $\kappa$, then $\kappa\cap D=\emptyset$, and $M$ carries an infinite-dimensional algebra of holomorphic germs. Indeed, $Z\setminus D$ is the preimage of an affine variety under a surjective morphism.
\item If every $K$-stable prime divisor contains $\kappa$, and the $K$-stabilisers of points of the open $K$-orbit are connected, then every holomorphic function on a connected neighbourhood of $M$ in $Z$ is constant.
\end{enumerate}
\end{thmx}

\emph{On the stabiliser hypothesis.} Theorem~\ref{thmB} and Theorem~\ref{thmC}(ii) are stated for general real forms, and there the hypothesis on $K_y$ is a genuine restriction. Flag manifolds $G/P$ are simply connected, so the issue is not the topology of $Z$ but the component group $K_y/K_y^0$. It can be non-trivial, and it depends on the choice of $K$ within its isogeny class. Two examples:
\begin{itemize}
\item For $\SL_3(\R)$ on the full flag manifold of $\C^3$, $K=\mathrm{SO}_3(\C)$ and the stabiliser of a generic flag is $(\mathbb Z/2)^2$.
\item For $\SU(p,p)$ on $\Gr_p(\C^{2p})$ with $K=S(\GL_p\times\GL_p)$, the stabiliser of the graph of an isomorphism is $\{g\in\GL_p:\det g=\pm1\}$.
\end{itemize}
Remark~\ref{rem:components} weakens the hypothesis to the requirement that every component of $K_y$ meet $K_0$. In the second example this weaker condition holds: the stabiliser has two components, and each meets $K_0=S(\U(p)\times\U(p))$, since a unitary $g$ with $\det g=-1$ commuting with $A^*A$ gives a unitary element $(g,AgA^{-1})$. For $\SU(p,q)$ on Grassmannians we avoid the issue altogether: we work with $K=\GL_p(\C)\times\GL_q(\C)$, for which every stabiliser is connected (Lemma~\ref{lem:connected}). Hence Theorems~\ref{thmD}--\ref{thmH} carry no such hypothesis.

\begin{thmx}[Main theorem for $\SU(p,q)$]\label{thmD}
Let $M=M(a,b,c)\subset Z=\Gr_k(\C^{p+q})$ be a Matsuki orbit of $\SU(p,q)$.
\begin{enumerate}[label=\textup{(\roman*)}]
\item If $M$ is not of holomorphic type, every holomorphic function on any connected neighbourhood of $M$ in $Z$ is constant.
\item If $M$ is of holomorphic type with $k=p$, $b=0$, $c\ge1$, and $(c,p,q)\neq(p,p,p)$, then in the chart $\cA$ the orbit $M$ is the manifold of partial isometries $\Ep\to\Em$ of rank $c$. A holomorphic hull of $M$ in $Z$ \textup{(}Definition~\ref{def:hull}\textup{)} is
\[
\widehat M=\{\Phi\in\Hom(\Ep,\Em):\ \rank\Phi\le c,\ \|\Phi\|\le1\}=\overline{\cD}\cap\{\rank\le c\}.
\]
Every function holomorphic near $M$ extends holomorphically to a neighbourhood of $\widehat M$ in $Z$, and no point of $Z\setminus\widehat M$ can be added.
\item If $c=p=q$, then $M\cong\U(p)$ is totally real, and it is its own hull.
\end{enumerate}
The case $a=0$, $k=q$ is symmetric.
\end{thmx}

\begin{thmx}[Regular functions on $K$-orbits]\label{thmE}
For the $K$-orbit $\kappa(r,s)=\{W:\dim W\cap\Ep=r,\ \dim W\cap\Em=s\}$ with $d=k-r-s\ge1$,
\[
\C[\kappa(r,s)]\ne\C\iff(s=0\text{ and }k=p)\ \text{or}\ (r=0\text{ and }k=q).
\]
That is, $\kappa(r,s)$ carries non-constant regular functions if and only if it is a rank stratum $\{\rank\Phi=k-r\}$ of a Harish-Chandra chart. Consequently, for $M$ not of holomorphic type and $c\ge1$, even functions holomorphic only on a neighbourhood of $M$ inside $\kappa$ are constant.
\end{thmx}

\begin{thmx}[Global CR functions; degree zero]\label{thmF}
Let $c\ge1$, let $CR^\infty(M)$ be the space of smooth CR functions on $M$, and let $CR^0(M)$ be the space of continuous CR functions \textup{(}in the sense of distributions\textup{)}.
\begin{enumerate}[label=\textup{(\roman*)}]
\item The $K_0$-finite elements of $CR^\infty(M)$ are exactly the restrictions of elements of $\C[\kappa]$.
\item $CR^\infty(M)=\C$ if and only if $M$ is not of holomorphic type.
\item The restriction map $\cO(\kappa)\to CR^\infty(M)$ is injective. It is bijective if and only if $M$ is not of holomorphic type.
\item $CR^0(M)=\C$ if and only if $M$ is not of holomorphic type.
\item If $M$ is of holomorphic type with $k=p$, $b=0$ and $(c,p,q)\ne(p,p,p)$, then $CR^0(M)$ is the uniform closure $\cP(M)$ of the holomorphic polynomials in $\Phi$. Its maximal ideal space is the hull $\widehat M=\overline\cD\cap\{\rank\le c\}$ of Theorem~\ref{thmD}, and restriction is an isometric isomorphism $\cP(\widehat M)\to CR^0(M)$. If $c=p=q$, then $CR^0(M)=C(M)$.
\end{enumerate}
\end{thmx}

\begin{thmx}[Line bundles]\label{thmG}
Let $M$ be not of holomorphic type, let $m\in\mathbb Z$, and let $U$ be a connected neighbourhood of $M$ in $Z$. Every holomorphic section of $\cO_Z(m)$ over $U$ is the restriction of a unique section in $H^0(Z,\cO_Z(m))$; in particular it vanishes if $m<0$.
\end{thmx}

\begin{thmx}[Cycle spaces]\label{thmH}
\begin{enumerate}[label=\textup{(\roman*)}]
\item Let $g\in G$ be such that $g\Ep$ is $h$-positive and $g\Em$ is $h$-negative definite. Then $g\kappa(a,b)\cap O(a,b,c)$ is the compact bundle of planes $W_+\oplus W_-\oplus N$ with $W_+\in\Gr_a(g\Ep)$, $W_-\in\Gr_b(g\Em)$ and $N\in\Iso_c((W_+\oplus W_-)^{\perp_h})$. Write $\cD_+$ for the set of $h$-positive $p$-planes and $\cD_-$ for the set of $h$-negative $q$-planes. Under $gK\mapsto(g\Ep,g\Em)$, the set of such $g$ corresponds to $\cD_+\times\cD_-$, which is the Akhiezer--Gindikin domain of $\SU(p,q)$ \cite{BHH03}, \cite{FHW}. So the cycle space of every $G_0$-orbit contains $\Om$.
\item The $K$-orbits of holomorphic type in the sense of Gindikin--Matsuki are the closed orbits $\kappa(k,0)=\Gr_k(\Ep)$ and $\kappa(0,k)=\Gr_k(\Em)$.
\item The Matsuki orbits of holomorphic type in the sense of Definition~\ref{def:holtype} with $c\ge1$ have non-closed dual $K$-orbits, and their cycle spaces equal $\Om$. The two notions of holomorphic type meet only in the base points $\{\Ep\}$ and $\{\Em\}$.
\end{enumerate}
\end{thmx}

Theorems~\ref{thmD} and~\ref{thmE} together give, for $\SU(p,q)$ on Grassmannians, a complete answer to the question of holomorphic functions and hulls of Matsuki orbits. Several consequences are worth stating.
\begin{itemize}
\item The sphere--ball picture does \emph{not} persist in general. For $\SU(2,6)$ on $\Gr_3(\C^8)$ the orbit $M(1,1,1)$ is a bundle of strictly pseudoconvex spheres $S^9$ over $\PP^1\times\PP^5$. One might guess that its hull is the corresponding ball bundle; in fact only constants are holomorphic near it. Its Levi form has signature $(5,1)$ with a four-dimensional kernel (Section~\ref{sec:examples}).
\item In the holomorphic-type case the hull contains points of rank $<c$, in particular the base cycle $\{\Ep\}$ of $\cD$. So the hull is \emph{not} contained in $\kappa$; nor does it contain $\kappa$, which is unbounded in $\cA$.
\item The holomorphic-type orbits are exactly the $K_0$-orbits of tripotents of rank $c$ in the Cartan factor $\Hom(\Ep,\Em)$. For these, $\widehat M$ coincides with the polynomial hull computed by Kaup and Zaitsev \cite[Thm.~2.4]{KZ03}, except when $c=p=q$. Our proof is independent and representation-theoretic, and gives extension of germs of holomorphic functions on neighbourhoods in $Z$.
\item The exception $c=p=q$ is a genuine discrepancy. There $M=\U(p)$ is the Shilov boundary of the matrix ball, so its polynomial hull is $\overline\cD$. But $\det(\Phi)^{-1}$ is holomorphic near $\U(p)$ and has poles along $\{\det=0\}$, which passes through the centre of $\cD$. So the holomorphic hull in $Z$ collapses to $M$ itself (Remark~\ref{rem:det}).
\item Germs and continuous CR functions behave differently. A germ along $M$ extends to a \emph{neighbourhood} of $\widehat M$ (Theorem~\ref{thmD}), whereas a continuous CR function extends only to $\widehat M$ itself, holomorphically in its interior, and in general to no neighbourhood of any point of $M$ (Example~\ref{ex:sqrt}).
\item For $c<\min(p-a,q-b)$ the constancy of CR functions also has a direct geometric reason: $M$ is swept out by chains of rational curves along its fibres (Proposition~\ref{prop:curves}).
\item Two dichotomies should not be confused.
\begin{itemize}
\item The function-theoretic dichotomy of Theorem~\ref{thmD} concerns $M$ itself.
\item The Gindikin--Matsuki dichotomy concerns the moving cycles $g\kappa\cap\gamma$.
\end{itemize}
They single out different orbits, and meet only at the base points $\{\Ep\}$ (when $k=p$) and $\{\Em\}$ (when $k=q$), by Theorem~\ref{thmH}. In particular, the holomorphic-type orbits of Theorem~\ref{thmD}(ii) with $c\ge1$ have the universal cycle space $\Om$, even though they carry many holomorphic functions.
\end{itemize}

\subsection*{Organisation}
Section~\ref{sec:lit} reviews the literature and isolates the gap. Section~\ref{sec:prelim} fixes notation. Section~\ref{sec:orbits} classifies the orbits and proves Theorem~\ref{thmA}. Section~\ref{sec:alg} proves Theorems~\ref{thmB} and~\ref{thmC}. Section~\ref{sec:main} proves Theorem~\ref{thmD} and Section~\ref{sec:regular} proves Theorem~\ref{thmE}. Section~\ref{sec:CR} proves Theorem~\ref{thmF}, including the description of continuous CR functions and the geometric proof by rational curves. Sections~\ref{sec:bundles} and~\ref{sec:cycles} prove Theorems~\ref{thmG} and~\ref{thmH}. Section~\ref{sec:levi} computes Levi forms in a family of examples, and Section~\ref{sec:examples} works out the examples.

%=====================================================================
\section{Literature review and the gap}\label{sec:lit}
%=====================================================================

\subsection{Real form orbits and Matsuki duality}
The systematic study of $G_0$-orbits on complex flag manifolds begins with Wolf \cite{Wolf69}. He proved finiteness, the existence of a unique closed orbit and of open orbits, and decomposed each orbit into holomorphic arc components; see Akhiezer's survey \cite{Akhiezer11} for a modern account. Matsuki \cite{Matsuki79,Matsuki88} established the duality between $G_0$- and $K$-orbits, and showed that it reverses closure relations and that dual orbits meet in a single $K_0$-orbit. Mirković, Uzawa and Vilonen \cite{MUV92} gave a Morse-theoretic proof for the full flag manifold, upgraded to an equivalence of equivariant sheaf categories, and Bremigan and Lorch \cite{BL02} extended the geometric proof to arbitrary parabolics. A complete exposition, including duality via momentum geometry, is \cite[Part~II]{FHW}.

\subsection{Cycle spaces}
For a flag domain $D$, the cycle space $\mathcal M_D$ was introduced in connection with automorphic cohomology \cite{WellsWolf77} and shown to be Stein \cite{Wolf92}. Its identification with the Akhiezer--Gindikin domain $\Om$ \cite{AG90}, outside explicit Hermitian cases, was achieved through Schubert slices, Iwasawa domains and Kobayashi hyperbolicity \cite{Huck02,HW03,FH05,WolfZierau00}; see \cite{FHW} for the complete theory and \cite{HW10} for exhaustions. For a lower-dimensional orbit $\gamma$, Gindikin and Matsuki \cite{GM03} and Huckleberry and Ntatin \cite{HN03} introduced the cycle space $\mathcal C(\gamma)$ of $g\in G$ with $g\kappa\cap\gamma$ non-empty and compact. Huckleberry and Ntatin proved that $\mathcal C(\gamma)=\Om$ for non-Hermitian $G_0$. Ntatin \cite{Ntatin} treated the Hermitian case, where a few exceptional orbits have the bounded symmetric domain as cycle space, and Matsuki \cite{Matsuki06} gave the $K$-orbit side. These works describe which $g$ keep the cycle compact, but not the analytic structure of $M$ itself.

\subsection{CR geometry of real form orbits and of Matsuki orbits}
Every $G_0$-orbit inherits a homogeneous CR structure. Altomani, Medori and Nacinovich developed the theory of these \emph{parabolic CR manifolds} via CR algebras. They characterised finite type and Levi non-degeneracy, constructed canonical and Mostow fibrations, computed fundamental groups \cite{AMN06,AMN10}, and studied global CR functions on minimal orbits \cite{AMN06,Altomani07}. They observed that a Matsuki orbit $\kappa\cap\gamma$ is a compact homogeneous CR manifold with the same CR codimension as $\gamma$ and smaller CR dimension \cite{AMN10}. In \cite{AMN13} they showed it is $\mathfrak n$-reductive, so that it fibres canonically over a complex flag manifold. Marini and Nacinovich \cite{MN17,MN18} related the tangential Cauchy--Riemann cohomology of such manifolds to the Dolbeault cohomology of their canonical complex embedding (the $K$-orbit $\kappa$) under pseudoconcavity hypotheses. Higher-order Levi forms and finite non-degeneracy of closed orbits were studied in \cite{MMN21,MMN23}.

\subsection{Bounded symmetric domains and compact orbits}
When $G_0$ is of Hermitian type and $Z$ is the compact dual of $\cD=G_0/K_0$, the boundary of $\cD$ in its Harish-Chandra realisation decomposes into finitely many $G_0$-orbits; see \cite[Ch.~3]{FHW} and \cite{Wolf69}. Kaup and Zaitsev studied the CR structure of these and related compact orbits in Jordan-theoretic terms. In \cite{KZ00} they introduced symmetric CR manifolds. These include the Shilov boundaries of bounded symmetric domains, for which they computed the polynomial and rational convex hulls and proved continuous extension of CR functions. In \cite{KZ03} they computed Levi forms and Levi cones of all orbits $K_0\cdot a$ in the ambient space of $\cD$, solved the CR-equivalence problem, and determined the polynomial hull of $K_0\cdot a$ as $\{z:\mu_j(z)\le\mu_j(a)\ \forall j\}$, where $\mu_j$ is the product of the $j$ largest singular values. They also proved that continuous CR functions on $K_0\cdot a$ extend to this hull. Local CR transformations of Levi-degenerate group orbits in compact Hermitian symmetric spaces are treated in \cite{KZ06}. Polynomial hulls of invariant compact sets in bounded symmetric domains were also studied by Sacr\'e \cite{Sacre92}, Faraut and Bouattour \cite{FB04} and Kaup \cite{Kaup04}. More generally, Gichev and Latypov \cite{GL01} showed that an orbit $K_0v$ of a compact group in a $K$-module is polynomially convex if and only if $Kv$ is closed and $K_0v$ is a real form of $Kv$.

\subsection{Holomorphic functions on invariant sets}
For $K_0$-invariant domains in complex $K$-spaces, holomorphic functions expand into $K_0$-isotypic components, which are often algebraic. Lassalle \cite{Lassalle78} developed the resulting Laurent theory in complexified compact symmetric spaces, and Heinzner \cite{Heinzner91} developed the theory of complexifications of $K_0$-actions on Stein spaces. The algebraic counterpart is the notion of an epimorphic subgroup $H\subset K$, characterised by $\C[K/H]=\C$ \cite{BienBorel92}.

\subsection{The gap}
To summarise:
\begin{itemize}
\item the orbit structure, the duality and the cycle spaces are well understood;
\item the CR structure of Matsuki orbits is known abstractly \cite{AMN10,AMN13};
\item the polynomial hulls of compact orbits are known in the special setting of the ambient vector space of a bounded symmetric domain \cite{KZ03,Kaup04}.
\end{itemize}
What is missing is a description of holomorphic functions and hulls of Matsuki orbits of lower-dimensional $G_0$-orbits in a general flag manifold. Concretely:
\begin{enumerate}[label=(G\arabic*)]
\item Which Matsuki orbits carry non-constant holomorphic functions on their neighbourhoods in $Z$, and what is the general mechanism?
\item When they do, what is the hull?
\item How do these orbits relate to the bounded symmetric domain and the Kaup--Zaitsev theory?
\item What are the global CR functions on $M$, and in degree zero, how does the CR cohomology of $M$ compare with the Dolbeault cohomology of $\kappa$ without pseudoconcavity hypotheses?
\item What are the holomorphic sections of line bundles near $M$, and how do the function-theoretic exceptions compare with the exceptions in cycle-space theory?
\end{enumerate}
Theorems~\ref{thmB} and~\ref{thmC} give the general mechanism in (G1) for arbitrary real forms. Theorems~\ref{thmD}--\ref{thmH} answer (G1)--(G5) completely for $\SU(p,q)$ on all Grassmannians.

%=====================================================================
\section{Preliminaries}\label{sec:prelim}
%=====================================================================

\subsection{Groups and the Grassmannian}
Let $n=p+q$, $\C^n=\Ep\oplus\Em$ with $\Ep=\la e_1,\dots,e_p\ra$, $\Em=\la e_{p+1},\dots,e_n\ra$, and
\[
h(z,w)=\sum_{i\le p}z_i\bar w_i-\sum_{j>p}z_j\bar w_j .
\]
Then $\pm h$ is the standard inner product on $E_\pm$ and $\Ep\perp_h\Em$. We write $\la\cdot,\cdot\ra$ and $\|\cdot\|$ for the standard inner product and norm, and $\|\Phi\|$ for the operator norm.

Let $G=\SL_n(\C)$, $G_0=\SU(p,q)$ and $Z=\Gr_k(\C^n)$, so that $T_WZ=\Hom(W,\C^n/W)$ and $\dim_\C Z=k(n-k)$. The maximal compact subgroup of $G_0$ is $S(\U(p)\times\U(q))$. Since scalar matrices act trivially on $Z$, it is convenient to use
\[
K_0=\U(\Ep)\times\U(\Em),\qquad K=\GL(\Ep)\times\GL(\Em),
\]
acting on $Z$ through $\GL_n(\C)$. These groups have the same image in $\mathrm{PGL}_n(\C)$ as $S(\U(p)\times\U(q))$ and $S(\GL_p\times\GL_q)$, hence the same orbits and the same notion of $K_0$-finite function. The advantage is that all stabilisers $K_W$ are connected (Lemma~\ref{lem:connected}). This fails for $S(\GL_p\times\GL_q)$. For example, if $p=q=k$ and $W$ is the graph of an isomorphism $A$, then the stabiliser is $\{(g,AgA^{-1}):\det(g)^2=1\}$, which has two components.

\subsection{Signatures and isotropic Grassmannians}
For $W\in Z$, the radical is $N_W=W\cap W^{\perp_h}$. The signature $\sig W=(a,b,c)$ records the maximal dimensions $a$, $b$ of $h$-positive and $h$-negative subspaces of $W$ and $c=\dim N_W$, so $a+b+c=k$. For a Hermitian space $V$ of signature $(m,m')$ let $\Iso_c(V)\subset\Gr_c(V)$ denote the set of totally isotropic $c$-planes.

\begin{lemma}\label{lem:Iso}
Let $V=V_+\oplus V_-$ be an $h$-orthogonal splitting into definite subspaces of dimensions $m,m'$. A $c$-plane $N\subset V$ is totally isotropic if and only if it is the graph of an isometric embedding $(U,h)\to(V_-,-h)$ with $U\in\Gr_c(V_+)$. Hence $\Iso_c(V)$ is compact, homogeneous under $\U(V_+)\times\U(V_-)$, and $\dim_\R\Iso_c(V)=2c(m+m'-c)-c^2$.
\end{lemma}

\begin{proof}
An isotropic $N$ meets $V_\pm$ trivially, so it is the graph of a map $A:U\to V_-$ with $U=\mathrm{pr}_+N$, and $0=h(u+Au)=\|u\|^2-\|Au\|^2$. Polarisation shows $A$ is isometric; the converse is the same computation. The group $\U(V_+)$ is transitive on $\Gr_c(V_+)$, and $\U(V_-)$ is transitive on isometric embeddings of a fixed $U$, a Stiefel manifold of real dimension $2cm'-c^2$.
\end{proof}

\subsection{Holomorphic hulls}
For a compact set $S$ in a complex manifold $Y$, $\cO(S)$ denotes the algebra of germs of holomorphic functions along $S$. Since $Z$ is compact, envelopes of holomorphy of neighbourhoods of $M$ are not the right objects (they need not exist as Riemann domains over $Z$, and constants extend everywhere). We use instead the following notion.

\begin{definition}\label{def:hull}
Let $S\subset Y$ be compact. A compact set $\widehat S\supset S$ is a \emph{holomorphic hull} of $S$ in $Y$ if
\begin{enumerate}[label=(H\arabic*)]
\item for every open $U\supset S$ and every $f\in\cO(U)$ there are an open set $\Omega\supset\widehat S$ and $F\in\cO(\Omega)$ such that $F=f$ on some open neighbourhood of $S$ contained in $U\cap\Omega$; and
\item for every $y\in Y\setminus\widehat S$ there are an open $U\supset S$ and $f\in\cO(U)$ with the following property: there is no connected open set $\Omega\supset\widehat S\cup\{y\}$ carrying some $F\in\cO(\Omega)$ with $F=f$ on a neighbourhood of $S$.
\end{enumerate}
\end{definition}

Both conditions concern germs along $S$: (H1) says that every germ extends to a neighbourhood of $\widehat S$, and (H2) says that no point outside $\widehat S$ can be added for all germs at once. We do not claim that a holomorphic hull is unique in general; in every case computed below, the set we exhibit is moreover the polynomial hull of $S$ in the chart $\cA$, or $S$ itself.

\begin{lemma}\label{lem:trivialhull}
Let $Y$ be compact and connected, and suppose $\cO(S)=\C$. Then $\widehat S=Y$ is the only holomorphic hull of $S$.
\end{lemma}

\begin{proof}
Every germ along $S$ is constant and therefore extends to $Y$, so (H1) holds for $\widehat S=Y$, and (H2) is vacuous. If $\widehat S\ne Y$, pick $y\in Y\setminus\widehat S$. Every $f$ as in (H2) is constant near $S$, and the constant extends to the connected open set $Y\supset\widehat S\cup\{y\}$; so (H2) fails.
\end{proof}

In this case we say that $S$ has \emph{trivial hull}. A non-trivial hull, i.e.\ one different from $Y$, can therefore exist only if $\cO(S)\ne\C$.

\begin{remark}[Relation to envelopes of holomorphy]\label{rem:envelope}
Definition~\ref{def:hull} is \emph{not} the classical envelope of holomorphy, and we never use it as such. The classical envelope of an open set $U\subset Y$ is the maximal Riemann domain $\pi:E(U)\to Y$ over $Y$ to which all of $\cO(U)$ extends. Rossi's envelope \cite{Rossi63} is the spectrum of $\cO(U)$. Both are defined in terms of $\cO(U)$, and for the orbits of Theorem~\ref{thmD}(i) both degenerate:
\begin{itemize}
\item $\cO(U)=\C$ for every connected neighbourhood $U$ of $M$;
\item every function on $U$ therefore extends to all of $Z$, so the maximal Riemann domain of extension is $Z$ itself;
\item the spectrum of $\cO(U)=\C$ is a single point;
\item in neither sense is the envelope the $K$-orbit $\kappa$, nor any other proper subset of $Z$.
\end{itemize}
The compact hull of Definition~\ref{def:hull} records exactly the information that survives: by Lemma~\ref{lem:trivialhull} it is non-trivial only if $\cO(M)\ne\C$. In that case $\widehat M$ is the maximal compact set to which all germs extend, and it controls the classical envelope as follows. By Theorem~\ref{thm:hull}, for every connected neighbourhood $U$ of $M$ there is an open set $\Omega_U\supset\widehat M$ such that every $f\in\cO(U)$ agrees, on the component of $U\cap\Omega_U$ containing $M$, with a function holomorphic on $\Omega_U$. Hence the Riemann domain of holomorphy $E(U)\to Z$ contains an open subset mapped biholomorphically onto $\Omega_U$, glued to $U$ along that component; in this sense $E(U)$ contains a neighbourhood of $\widehat M$. By (H2), no point of $Z\setminus\widehat M$ is reached for all $U$ simultaneously. For the uniform algebra $CR^0(M)$ of continuous CR functions, the analogue of Rossi's envelope is its maximal ideal space; by Theorem~\ref{thmF}(v) it is again $\widehat M$.
\end{remark}

Condition (H2) will be verified by the following standard device. Suppose $\Omega$ is a connected neighbourhood of $\widehat S\cup\{y\}$, and $F$ is holomorphic on an open set $\Omega'\subset\Omega$ with $\Omega\setminus\Omega'$ a proper analytic subset, so that $\Omega'$ is connected and dense. Suppose $F$ is holomorphic near $\widehat S$ but unbounded near $y$, or $F$ is non-vanishing near $\widehat S$ while $F(y)=0$. Then $f=F$, respectively $f=1/F$, regarded as a germ along $S$, admits no holomorphic extension to $\Omega$. Indeed, an extension agreeing with $f$ near $S$ would agree with $f$ on the connected set $\Omega'$ (respectively $\Omega'\setminus\{F=0\}$, which is again connected), by the identity theorem; this contradicts the unboundedness of $F$ near $y$, respectively $F(y)=0$.

\subsection{Relation to fibrewise extension domains}\label{subsec:earlier}
A related work of the author and K.~Shabbir concerns the same class of orbits, and we explain how it relates to the present results.

\emph{The paper \cite{UllahShabbir}} treats the rank-one orbits $M(a,b,1)$ with $\min(p-a,q-b)=1$ and $s=\max(p-a,q-b)\ge2$. Each fibre of $M\to B=\Gr_a(\Ep)\times\Gr_b(\Em)$ is then a strictly pseudoconvex sphere $\Iso_1(V')\subset\PP(V')$ bounding a projective ball. By \cite[Thm.~1]{UllahShabbir}, every real-analytic CR function on $M$ extends holomorphically to the union $D_{a,b,1}$ of these balls. This union is a domain in a projective bundle $X_{a,b}\to B$, not in $Z$, and it is maximal among domains obtained fibre by fibre. That paper leaves open whether the extensions descend to the image of $D_{a,b,1}$ under the evaluation map $\mathrm{ev}:X_{a,b}\to Z$. The present results settle this question.
\begin{itemize}
\item If $M$ is not of holomorphic type, for instance $M(1,1,1)\subset\Gr_3(\C^8)$, every continuous CR function on $M$ is constant (Theorem~\ref{thmF}(iv)). So the extensions of \cite[Thm.~1]{UllahShabbir} are constant and descend trivially, and $M$ has trivial hull $Z$ (Lemma~\ref{lem:trivialhull}) (Example~\ref{ex:ballbundle}).
\item The rank-one orbits of holomorphic type are $M(p-1,0,1)$ with $k=p$ and $M(0,q-1,1)$ with $k=q$. Consider the first; the second is symmetric. A point of the closed ball over $W_+\in\Gr_{p-1}(\Ep)$ is a line $[u+v]$ with $u$ spanning $\Ep\ominus W_+$, $v\in\Em$ and $\|v\|\le\|u\|$. The evaluation map sends it to the graph of the map $\Phi$ with $\Phi|_{W_+}=0$ and $\Phi u=v$, a contraction of rank $\le1$. So $\mathrm{ev}$ maps $\overline{D_{p-1,0,1}}$ onto the hull $\widehat M=\overline\cD\cap V_1$, collapsing the centre section $\{[u]\}\cong\PP^{p-1}$ to the base point $\Ep$. Now let $f$ be a CR function on $M$. By Theorem~\ref{thmF}(v), $f$ is the restriction of some $\tilde f\in\cP(\widehat M)$. Then $\tilde f\circ\mathrm{ev}$ is a uniform limit of the holomorphic functions $P_n\circ\mathrm{ev}$, so it is holomorphic on $D_{p-1,0,1}$, continuous up to the boundary, and has boundary values $f\circ\mathrm{ev}$. By the fibrewise uniqueness in \cite[Thm.~1]{UllahShabbir}, it is the extension constructed there. Hence the fibrewise extensions descend to $Z$, and their images fill out exactly the hull $\widehat M$.
\end{itemize}

%=====================================================================
\section{Orbits and the structure of Matsuki orbits}\label{sec:orbits}
%=====================================================================

\subsection{\texorpdfstring{$G_0$}{G0}-orbits}

\begin{proposition}\label{prop:G0}
Two $k$-planes lie in the same $G_0$-orbit if and only if they have the same signature. The orbits are
\[
O(a,b,c)=\{W:\sig W=(a,b,c)\},\qquad a+b+c=k,\quad a+c\le p,\quad b+c\le q .
\]
The map $W\mapsto N_W$ is a $G_0$-equivariant fibre bundle $O(a,b,c)\to\Iso_c(\C^n)$ whose fibre is an open orbit of $\U(p-c,q-c)$ in $\Gr_{k-c}(\C^{n-2c})$. Consequently $\codim_\R O(a,b,c)=c^2$: the orbit is open iff $c=0$ and closed iff $c$ is maximal. Every $G_0$-orbit is a generic CR submanifold of $Z$.
\end{proposition}

\begin{proof}
If $\sig W=\sig W'$, Sylvester's law of inertia and Witt's extension theorem give an element of $\U(p,q)$ mapping $W$ to $W'$; rescaling by a scalar puts it in $\SU(p,q)$ without changing its action on $Z$.

For the inequalities: $W$ contains the $(a+c)$-dimensional $h$-semipositive subspace $W_+\oplus N_W$, which meets $\Em$ trivially, so $a+c\le p$; similarly $b+c\le q$. Conversely,
\[
\la e_1,\dots,e_a\ra\oplus\la e_{p+1},\dots,e_{p+b}\ra\oplus\la e_{a+i}+e_{p+b+i}:1\le i\le c\ra
\]
has signature $(a,b,c)$.

For the fibration: $W\mapsto W/N_W$ identifies $\{W:N_W=N\}$ with the non-degenerate $(k-c)$-planes of signature $(a,b)$ in $N^{\perp_h}/N$, a space of signature $(p-c,q-c)$, on which the stabiliser of $N$ acts transitively. By Lemma~\ref{lem:Iso}, $\dim_\R\Iso_c(\C^n)=2c(n-c)-c^2$. Adding $2(k-c)(n-k-c)$ gives $2k(n-k)-c^2$.

Genericity: at $W$ the tangent space of the orbit is the image of $\mathfrak g_0$ in $T_WZ=\mathfrak g/\mathfrak p_W$, and $\mathfrak g_0+i\mathfrak g_0=\mathfrak g$. So $T\gamma+J\,T\gamma=TZ$.
\end{proof}

\subsection{\texorpdfstring{$K$}{K}-orbits and their closures}

\begin{proposition}\label{prop:K}
\begin{enumerate}[label=\textup{(\roman*)}]
\item The $K$-orbits are the sets $\kappa(r,s)=\{W:\dim W\cap\Ep=r,\ \dim W\cap\Em=s\}$, where $d=k-r-s$ satisfies $d\le p-r$ and $d\le q-s$.
\item Put $V_1=\Ep/W_+$ and $V_2=\Em/W_-$, of dimensions $m_1=p-r$ and $m_2=q-s$. The map
\[
\pi:\kappa(r,s)\to B:=\Gr_r(\Ep)\times\Gr_s(\Em),\qquad W\mapsto(W\cap\Ep,W\cap\Em),
\]
is a $K$-equivariant holomorphic fibre bundle. Its fibre over $(W_+,W_-)$ is
\[
F=\{S\in\Gr_d(V_1\oplus V_2):S\cap V_1=0=S\cap V_2\},
\]
the set of graphs of injective maps $A:U\to V_2$ with $U\in\Gr_d(V_1)$.
\item $\dim_\C\kappa(r,s)=r(p-r)+s(q-s)+d(p-r-d)+dm_2$.
\end{enumerate}
In particular $(a,b,c)\mapsto(r,s)=(a,b)$ is a bijection between the index sets of $G_0$-orbits and of $K$-orbits.
\end{proposition}

\begin{proof}
$W/(W_+\oplus W_-)$ embeds into $V_1\oplus V_2$ and meets neither summand, by maximality of $W_\pm$. The group $K$ is transitive on pairs $(W_+,W_-)$, and the stabiliser acts on $V_1\oplus V_2$ through $\GL(V_1)\times\GL(V_2)$, which is transitive on $F$. The constraints say that $U\subset V_1$ and that $A$ is injective. The dimension is $\dim B+\dim\Gr_d(V_1)+d\,m_2$.
\end{proof}

\begin{lemma}\label{lem:connected}
Every stabiliser $K_W$ is connected.
\end{lemma}

\begin{proof}
Let $\pi_\pm$ be the projections of $\C^n=\Ep\oplus\Em$, and put $W_\pm=W\cap E_\pm$ and $P_\pm=\pi_\pm(W)$, so that $W_\pm\subset P_\pm\subset E_\pm$.
\begin{itemize}
\item The quotient $W/(W_+\oplus W_-)$ is the graph of an isomorphism $\bar A:P_+/W_+\to P_-/W_-$.
\item $W$ is the preimage of this graph in $P_+\oplus P_-$.
\end{itemize}
Hence $(g_+,g_-)\in K$ preserves $W$ if and only if two conditions hold:
\begin{itemize}
\item $g_\pm$ preserves the flag $W_\pm\subset P_\pm$, i.e.\ $g_\pm$ lies in the corresponding parabolic subgroup $Q_\pm\subset\GL(E_\pm)$;
\item the induced maps $\bar g_\pm$ on $P_\pm/W_\pm$ satisfy $\bar g_-\bar A=\bar A\bar g_+$.
\end{itemize}
Consider the surjective homomorphism
\[
\rho:Q_+\times Q_-\to\GL(P_+/W_+)\times\GL(P_-/W_-).
\]
Then $K_W=\rho^{-1}(H)$, where $H=\{(x,\bar Ax\bar A^{-1})\}\cong\GL(P_+/W_+)$. The groups involved are connected:
\begin{itemize}
\item $H$ is connected;
\item $\ker\rho$ consists of block-triangular matrices whose middle diagonal block is the identity, and is isomorphic to a product of general linear groups and a unipotent group, hence connected.
\end{itemize}
An extension of a connected group by a connected group is connected, so $K_W$ is connected.

Note that $K_W$ is \emph{not} a parabolic subgroup of $K$ in general. For instance, if $p=q=k$ and $W$ is the graph of an isomorphism, then $K_W\cong\GL_p$, which has dimension smaller than a Borel subgroup of $K$. Stabilisers are parabolic only for points of closed $K$-orbits.
\end{proof}

\begin{lemma}[Closure relations]\label{lem:closure}
For admissible pairs, $\kappa(a,b)\subset\overline{\kappa(r,s)}$ if and only if $r\le a$ and $s\le b$.
\end{lemma}

\begin{proof}
Necessity follows from upper semicontinuity of $W\mapsto\dim W\cap E_\pm$.

For sufficiency, let $W\in\kappa(a,b)$, written as in the proof of Lemma~\ref{lem:connected} with $d=k-a-b$ and $\dim U'=d$. Choose splittings $W_+=W_+'\oplus X$ and $W_-=W_-'\oplus Y$ with $\dim W_+'=r$ and $\dim W_-'=s$. Put $d'=k-r-s=d+(a-r)+(b-s)$. Admissibility of $(r,s)$ gives $d'+s\le q$ and $d'+r\le p$, that is
\[
(a-r)+\dim(W_-+A(U'))\le q,\qquad (b-s)+\dim(W_++U')\le p .
\]
Hence there are injective maps $\varphi:X\to\Em$ and $\psi:Y\to\Ep$ with
\[
\varphi(X)\cap(W_-+A(U'))=0,\qquad \psi(Y)\cap(W_++U')=0 .
\]
For $t\in\C$ put
\[
W_t=W_+'\oplus W_-'\oplus\{x+t\varphi x\}\oplus\{y+t\psi y\}\oplus\{u+Au\}.
\]
The five summands are independent for all $t$, so $W_t$ is a $k$-plane and $W_0=W$.

Let $t\ne0$ and suppose $w=w_+'+w_-'+(x+t\varphi x)+(y+t\psi y)+(u+Au)\in\Ep$. Its $\Em$-component $w_-'+t\varphi x+y+Au$ must vanish. Since $W_-=W_-'\oplus Y$, and $\varphi(X)$, $A(U')$, $W_-$ are independent, this forces $w_-'=0$, $y=0$, $x=0$ and $u=0$. Hence $W_t\cap\Ep=W_+'$; symmetrically $W_t\cap\Em=W_-'$. So $W_t\in\kappa(r,s)$ for $t\ne0$, and $W\in\overline{\kappa(r,s)}$.
\end{proof}

\subsection{Matsuki duality and the structure theorem}

\begin{theorem}[Theorem~\ref{thmA}]\label{thm:structure}
For every admissible $(a,b,c)$, the $K$-orbit $\kappa(a,b)$ is Matsuki-dual to $O(a,b,c)$, and
\begin{align*}
M(a,b,c):={}&\kappa(a,b)\cap O(a,b,c)\\
={}&\bigl\{W_+\oplus W_-\oplus N:\ W_+\in\Gr_a(\Ep),\ W_-\in\Gr_b(\Em),\ N\in\Iso_c\bigl((W_+\oplus W_-)^{\perp_h}\bigr)\bigr\}.
\end{align*}
\begin{enumerate}[label=\textup{(\roman*)}]
\item The restriction of $\pi$ makes $M$ a $K_0$-homogeneous bundle over $\Gr_a(\Ep)\times\Gr_b(\Em)$ with fibre $\Iso_c(\C^{p-a,q-b})$, the closed $\SU(p-a,q-b)$-orbit in $\Gr_c(\C^{p+q-a-b})$.
\item $\dim_\R M=2a(p-a)+2b(q-b)+2c(p-a+q-b-c)-c^2$. The $K_0$-isotropy at $W_+\oplus W_-\oplus N$ is $\U(W_+)\times\U(W_-)\times\U(c)_\Delta\times\U(p-a-c)\times\U(q-b-c)$.
\item $M$ is a generic CR submanifold of $\kappa(a,b)$ of real codimension, and hence CR codimension, $c^2$.
\item The radical map $\nu:M\to\Iso_c(\C^n)$, $W\mapsto N_W$, is a $K_0$-equivariant bundle with compact complex fibres $\Gr_a(\Ep\cap N^{\perp_h})\times\Gr_b(\Em\cap N^{\perp_h})\cong\Gr_a(\C^{p-c})\times\Gr_b(\C^{q-c})$.
\end{enumerate}
\end{theorem}

\begin{proof}
\emph{Description of $M$.} Let $W\in\kappa(a,b)\cap O(a,b,c)$ and put $W_\pm=W\cap E_\pm$, $N=N_W$. Then $N\perp_hW_+\oplus W_-$. On $W_+\oplus W_-$, $h(w_++w_-)=\|w_+\|^2-\|w_-\|^2$, so the only vector of $W_+\oplus W_-$ that is $h$-orthogonal to all of $W_+\oplus W_-$ is $0$. Hence $(W_+\oplus W_-)\cap N=0$, and by dimension $W=W_+\oplus W_-\oplus N$.

Conversely, let $W=W_+\oplus W_-\oplus N$ as in the statement. Its Gram matrix is $\operatorname{diag}(I_a,-I_b,0_c)$ in a suitable basis, so $\sig W=(a,b,c)$. If $w_++w_-+\nu\in\Ep$ with $\nu=\nu_++\nu_-$, then $w_-+\nu_-=0$. Now $\nu\perp_hW_-$ means $\nu_-\perp W_-$ in $\Em$, and $\nu_-=-w_-\in W_-$, so $\nu_-=0=w_-$. Then $\nu=\nu_+\in\Ep$ is isotropic, so $\nu=0$. Thus $W\cap\Ep=W_+$, and symmetrically $W\cap\Em=W_-$.

\emph{Homogeneity, compactness, duality.} $K_0$ is transitive on pairs $(W_+,W_-)$, and $\U(\Ep\ominus W_+)\times\U(\Em\ominus W_-)$ is transitive on $\Iso_c$ of the complement (Lemma~\ref{lem:Iso}). So the intersection is one compact $K_0$-orbit, and $\kappa(a,b)$ is dual to $O(a,b,c)$ by Matsuki's characterisation \cite{Matsuki88}, \cite[Ch.~8]{FHW}. This proves (i).

(ii) $\dim K_0-\dim L_0=p^2+q^2-\bigl(a^2+b^2+c^2+(p-a-c)^2+(q-b-c)^2\bigr)$, which simplifies to the stated value.

(iii) $T_WM$ is the image of $\mathfrak k_0$ and $T_W\kappa$ that of $\mathfrak k=\mathfrak k_0+i\mathfrak k_0$, so $M$ is generic in $\kappa$. By Proposition~\ref{prop:K}(iii), $2\dim_\C\kappa(a,b)-\dim_\R M=c^2$.

(iv) Given isotropic $N$, a plane $W_+\oplus W_-\oplus N$ lies in $M$ iff $W_\pm\subset E_\pm\cap N^{\perp_h}$. These spaces have dimensions $p-c$ and $q-c$, since $N$ projects injectively to $E_\mp$. The embedding $(W_+,W_-)\mapsto W_+\oplus W_-\oplus N$ is holomorphic.
\end{proof}

\begin{remark}
For $c=0$, $M=\Gr_a(\Ep)\times\Gr_b(\Em)$ is the base cycle of the flag domain $O(a,b,0)$; for maximal $c$, $M=\gamma$ is the closed orbit. Since $T\kappa$ is complex, $H(M)=H(\gamma)\cap T\kappa$. Part (iii) agrees with the general result of \cite{AMN10,AMN13} on the CR codimension of Matsuki orbits.
\end{remark}

%=====================================================================
\section{Algebraisation and the divisor criterion}\label{sec:alg}
%=====================================================================

In this section $G_0$, $Z=G/P$, $K_0$ and $K=K_0^\C$ are general, with $K_0$ compact and connected; $K$ is then a connected reductive algebraic group. For a $K$-orbit $\kappa'=Ky\cong K/K_y$, $\C[\kappa']=\C[K]^{K_y}$ denotes its algebra of regular functions.

\begin{theorem}[Theorem~\ref{thmB}]\label{thm:alg}
Let $y\in Z$ with $K_y$ connected, $U\supset K_0y$ a $K_0$-invariant open set, and $f\in\cO(U)$ $K_0$-finite. There is a unique $F\in\C[Ky]$ with $F=f$ on a neighbourhood of $K_0y$ in $Ky$. Conversely, every $F\in\C[Ky]$ is $K_0$-finite.
\end{theorem}

\begin{proof}
Let $V\subset\cO(U)$ be the finite-dimensional span of the translates $f\circ k^{-1}$, $k\in K_0$, with $(\rho(k)g)(z)=g(k^{-1}z)$. A finite-dimensional continuous representation of the compact connected group $K_0$ extends uniquely to a rational representation of $K$ \cite[Ch.~III]{BtD}.

For $X\in\mathfrak k_0$ write $\xi_X(z)=\frac{d}{dt}|_{0}\exp(tX)z$; then $d\rho(X)g=-\xi_Xg$. Since $K$ acts holomorphically, the vector field of $X+iY$ is $\xi_X+J\xi_Y$, and $(Jv)g=i\,vg$ for holomorphic $g$. By complex linearity, $d\rho(\zeta)g=-\xi_\zeta g$ for all $\zeta\in\mathfrak k$. For $\zeta\in\mathfrak k_y$, $\xi_\zeta(y)=0$, so $\mathrm{ev}_y\circ d\rho(\zeta)=0$ on $V$. Since $V$ is $\rho$-stable, $\frac{d}{dt}\,\mathrm{ev}_y\circ\rho(\exp t\zeta)=\mathrm{ev}_y\circ d\rho(\zeta)\circ\rho(\exp t\zeta)=0$, so $\mathrm{ev}_y\circ\rho(h)=\mathrm{ev}_y$ for $h\in\exp(\mathfrak k_y)$. These elements generate the identity component $K_y^0$, and $K_y^0=K_y$ by hypothesis; this is the only place where connectedness is used (see Remark~\ref{rem:components}).

Define $F(ky)=(\rho(k^{-1})f)(y)$. This is well defined by the previous paragraph; it is regular, being a matrix coefficient of a rational representation; and it is right $K_y$-invariant.

Let $\Omega_0$ be the connected component of $\{k\in K:ky\in U\}$ containing $K_0$. On $\Omega_0$ the holomorphic functions $k\mapsto f(ky)$ and $k\mapsto F(ky)$ agree on $K_0$, which is a totally real submanifold of maximal dimension. Hence they agree on $\Omega_0$, and $F=f$ on the open set $\Omega_0y\supset K_0y$ of $Ky$. Uniqueness follows from the identity theorem on the connected manifold $Ky$. The converse holds because $\C[Ky]$ is a rational, hence locally finite, $K$-module.
\end{proof}

\begin{remark}[Disconnected stabilisers]\label{rem:components}
Connectedness of $K_y$ enters the proof only through the identity $\mathrm{ev}_y\circ\rho(h)=\mathrm{ev}_y$ on $V$ for $h\in K_y$. For $h\in(K_0)_y$ this holds trivially, since $(\rho(h)g)(y)=g(h^{-1}y)=g(y)$. The Lie algebra argument gives it for $h\in K_y^0$. Hence Theorem~\ref{thm:alg} holds whenever
\[
K_y=K_y^0\cdot(K_0)_y,
\]
that is, whenever every component of $K_y$ meets $K_0$. Without this condition the proof only yields a regular function on the finite cover $K/K_y^0(K_0)_y$ of $Ky$, which need not descend to $Ky$. So the hypothesis cannot simply be dropped from the argument, and it is the reason for the hypothesis in Theorem~\ref{thm:divisor}(ii). Theorem~\ref{thm:divisor}(ii) uses Theorem~\ref{thm:alg} only at points of $\kappa_0$ near $M$, so it holds whenever the weaker condition is satisfied at those points.

Two examples:
\begin{itemize}
\item For $\SL_{n+1}(\R)$ on $\PP^n$ (Example~\ref{ex:general}), the stabiliser of a real non-isotropic line is $S(\mathrm O_1\times\mathrm O_n)(\C)$. It is disconnected, but each component contains a real diagonal matrix, so the weaker condition holds.
\item For $\SL_3(\R)$ on full flags, the stabiliser $(\mathbb Z/2)^2$ of a real generic flag lies in $\mathrm{SO}_3$. At non-real generic flags it need not meet $K_0$.
\end{itemize}
\end{remark}

\begin{corollary}[Fourier expansion]\label{cor:fourier}
Let $U$ be a connected $K_0$-invariant open set and $f\in\cO(U)$. For an irreducible $K_0$-type $\delta$ of degree $d_\delta$ and character $\chi_\delta$, put $f_\delta(z)=d_\delta\int_{K_0}\overline{\chi_\delta(k)}f(k^{-1}z)\,dk$. Then:
\begin{enumerate}[label=\textup{(\roman*)}]
\item $f_\delta\in\cO(U)$ is $K_0$-finite of type $\delta$, and $\sup_L|f_\delta|\le d_\delta^2\sup_L|f|$ for every $K_0$-invariant $L\subset U$;
\item if $f_\delta=0$ for all $\delta\ne1$, then $f$ is $K_0$-invariant and $f=f_1$;
\item if $\sum_\delta f_\delta$ converges absolutely and locally uniformly on a $K_0$-invariant open set $U'\subset U$, its sum is $f$ there.
\end{enumerate}
\end{corollary}

\begin{proof}
For $z\in U$, the $\delta$-isotypic part of the continuous function $\varphi_z(k)=f(k^{-1}z)$ on $K_0$ is $k\mapsto f_\delta(k^{-1}z)$, because $\chi_\delta$ is a class function. Now apply the Peter--Weyl theorem and $|\chi_\delta|\le d_\delta$.
\end{proof}

\begin{proposition}[Codimension criterion]\label{thm:codim}
Let $M=\kappa\cap\gamma$ be a Matsuki orbit, $\kappa_0$ the open $K$-orbit, and
\[
Y_\kappa=\bigcup\{\kappa':\ \kappa\not\subset\overline{\kappa'}\}.
\]
Assume $K_y$ is connected for $y\in\kappa_0$. Then $Y_\kappa$ is a closed algebraic subset, and $\Omega=K\cdot U\supset Z\setminus Y_\kappa$ for every neighbourhood $U$ of $M$. If $\codim_\C Y_\kappa\ge2$, then $\cO(U)=\C$ for every connected neighbourhood $U$ of $M$ in $Z$.
\end{proposition}

\begin{proof}
\emph{$Y_\kappa$ is closed.} If $\kappa''\subset\overline{\kappa'}$ and $\kappa'\subset Y_\kappa$, then $\kappa''\subset Y_\kappa$, since otherwise $\kappa\subset\overline{\kappa''}\subset\overline{\kappa'}$. So $Y_\kappa$ is a finite union of orbit closures.

\emph{$K\cdot U\supset Z\setminus Y_\kappa$.} If $\kappa\subset\overline{\kappa'}$, then $\kappa'$ meets every neighbourhood of every point of $M$, hence meets $U$, so $\kappa'\subset KU$.

\emph{Vanishing.} Shrinking $U$, we may assume $U$ is connected and $K_0$-invariant; constancy then propagates to the original connected neighbourhood by the identity theorem. Let $f\in\cO(U)$ and fix $\delta$. Let $V_\delta\subset\cO(U)$ be the finite-dimensional $K_0$-span of $f_\delta$. The set $U\cap\kappa_0$ is non-empty and connected, since $Z\setminus\kappa_0$ is a proper analytic subset.

By Theorem~\ref{thm:alg}, applied at points of $U\cap\kappa_0$ and combined with the identity theorem on $U\cap\kappa_0$, each $g\in V_\delta$ equals on $U\cap\kappa_0$ a unique $G_g\in\C[\kappa_0]$. The map $g\mapsto G_g$ is linear, injective and $K_0$-equivariant. Its image is a $K_0$-stable finite-dimensional subspace of the rational $K$-module $\C[\kappa_0]$, hence $K$-stable, because $K_0$ is Zariski dense in $K$.

Thus for $k\in K$ there is $g'\in V_\delta$ with $G_{f_\delta}\circ k^{-1}=G_{g'}$, so $G_{f_\delta}=g'\circ k$ on $k^{-1}U\cap\kappa_0$. Hence $G_{f_\delta}$ extends holomorphically to $k^{-1}U$, and these extensions agree on overlaps, which contain the dense set $\kappa_0$. So $G_{f_\delta}$ extends to a holomorphic function on $KU\supset Z\setminus Y_\kappa$.

If $\codim Y_\kappa\ge2$, the Riemann extension theorem across analytic sets of codimension $\ge2$ \cite[Ch.~7]{GrauertRemmert} extends it to the compact connected manifold $Z$, where it is constant. Therefore $f_\delta$ is constant. A constant is $K_0$-invariant, so $f_\delta=0$ for $\delta\ne1$, and $f=f_1$ is constant by Corollary~\ref{cor:fourier}(ii).
\end{proof}

\begin{lemma}\label{lem:divisors}
$\codim_\C Y_\kappa\ge2$ if and only if every $K$-stable prime divisor of $Z$ contains $\kappa$. Moreover, a $K$-stable closed subset $D$ either contains $\kappa$ or is disjoint from it.
\end{lemma}

\begin{proof}
The last statement holds because $\kappa$ is a single orbit. The irreducible components of $Y_\kappa$ are orbit closures $\overline{\kappa'}$ with $\kappa\not\subset\overline{\kappa'}$. A component of codimension one is therefore a $K$-stable prime divisor not containing $\kappa$. Conversely, a $K$-stable prime divisor $D$ is the closure of its dense $K$-orbit $\kappa'$, since there are finitely many orbits. If $\kappa\not\subset D$, then $\kappa'\subset Y_\kappa$, so $D\subset Y_\kappa$ and $\codim Y_\kappa=1$.
\end{proof}

\begin{theorem}[Theorem~\ref{thmC}]\label{thm:divisor}
Let $M=\kappa\cap\gamma$ be a Matsuki orbit in a $G$-flag manifold $Z$.
\begin{enumerate}[label=\textup{(\roman*)}]
\item If a $K$-stable prime divisor $D$ does not contain $\kappa$, then $M\subset Z\setminus D$. There is a surjective morphism $\varphi$ from $Z$ onto a projective variety $Z''$ and a hyperplane section $H$ of $Z''$ with $Z\setminus D=\varphi^{-1}(Z''\setminus H)$, where $Z''\setminus H$ is affine of positive dimension. In particular $\cO(M)$ is infinite-dimensional.
\item If every $K$-stable prime divisor contains $\kappa$, and $K_y$ is connected for $y$ in the open $K$-orbit, then $\cO(U)=\C$ for every connected neighbourhood $U$ of $M$ in $Z$.
\end{enumerate}
\end{theorem}

\begin{proof}
(ii) follows from Lemma~\ref{lem:divisors} and Proposition~\ref{thm:codim}.

(i) By Lemma~\ref{lem:divisors}, $\kappa\cap D=\emptyset$, so $M\subset Z\setminus D$. On a flag manifold every line bundle with a non-zero section is globally generated \cite{Brion05}. Let $L=\cO(D)$ with canonical section $s_D$, and let $\varphi:Z\to\PP(H^0(Z,L)^*)$ be the morphism defined by $L$, with image $Z''$. Then $s_D=\varphi^*\sigma$ for a linear form $\sigma$, and $D=\varphi^{-1}(H)$ with $H=Z''\cap\{\sigma=0\}$. So $Z\setminus D=\varphi^{-1}(Z''\setminus H)$, and $Z''\setminus H$ is affine.

Since $D$ is non-empty and $L$ is globally generated, $\dim H^0(Z,L)\ge2$, so $Z''$ is not a point. Choose a non-constant regular function $g$ on the irreducible affine variety $Z''\setminus H$. Then $g\circ\varphi$ is regular on $Z\setminus D$. It is non-constant on every non-empty open subset: otherwise, by the identity theorem on the connected manifold $Z\setminus D$, $g$ would be constant on $\varphi(Z\setminus D)=Z''\setminus H$. Hence the powers $(g\circ\varphi)^j$, $j\ge0$, are linearly independent germs along $M$.
\end{proof}

\begin{example}\label{ex:general}
\begin{itemize}
\item For $\SU(p,q)$ on $\PP^{n-1}$ and the closed orbit $\Sigma=\{h=0\}$, $\kappa$ is the open $K$-orbit. The $K$-stable proper subvarieties are $\PP(\Ep)$ and $\PP(\Em)$, of codimensions $q$ and $p$. So $\Sigma$ carries non-constant germs iff $p=1$ or $q=1$, when $\Sigma$ bounds a ball.
\item For the non-Hermitian real form $\SL_{n+1}(\R)$ on $\PP^n$, $K=\mathrm{SO}_{n+1}(\C)$ and the quadric $Q$ of isotropic lines is a $K$-stable prime divisor. The closed orbit $\R\PP^n$ is the Matsuki orbit of the open $K$-orbit $\PP^n\setminus Q$. By Theorem~\ref{thm:divisor}(i), $\cO(\R\PP^n)$ is infinite-dimensional. So non-constant germs occur also for non-Hermitian real forms; there they arise from totally real Matsuki orbits.
\end{itemize}
\end{example}

\begin{proposition}\label{prop:RPn}
For $\SL_{n+1}(\R)$ acting on $\PP^n$, the totally real orbit $\R\PP^n$ is its own holomorphic hull in $\PP^n$, and $\cO(\R\PP^n)$ is infinite-dimensional.
\end{proposition}

\begin{proof}
Let $q(x)=\sum x_j^2$, so that $Q=\{q=0\}$. The map $[x]\mapsto xx^{T}/q(x)$ is a $K$-equivariant closed embedding of the affine variety $\PP^n\setminus Q$ into the space $\mathrm{Sym}_{n+1}(\C)$ of complex symmetric matrices. Its image is $\{X:X^2=X,\ X^T=X,\ \operatorname{tr}X=1\}$.

The real points of the image are the real symmetric rank-one projections. These are exactly the images of points of $\R\PP^n$, and they form a compact subset $C\subset\mathrm{Sym}_{n+1}(\R)$.

Infinite-dimensionality follows from Theorem~\ref{thm:divisor}(i). Condition (H1) is trivial, so it remains to verify (H2). Let $y\in\PP^n\setminus\R\PP^n$; there are two cases.

\emph{Case $y\in Q$.} Choose a real linear form $\ell$ with $\ell(y)\ne0$. The function $\ell^2/q$ is holomorphic on $\PP^n\setminus Q$, hence near $\R\PP^n$, where $q>0$. It is unbounded near $y$. Apply the device after Definition~\ref{def:hull} with $\Omega'=\Omega\setminus Q$.

\emph{Case $y\notin Q$.} The image $Y$ of $y$ is not real, so $Y=A+iB$ with $A,B$ real symmetric and $B\ne0$. Put
\[
h(X)=\exp\big(-\operatorname{tr}((X-A)^2)\big).
\]
This is an entire function. For real $X$ we have $|h(X)|=\exp(-\|X-A\|^2)\le1$, while $h(Y)=\exp(\|B\|^2)>1$. The Taylor series of the entire function $h$ converges uniformly on the compact set $C\cup\{Y\}$. Choosing a Taylor polynomial $P$ with $|P-h|<\tfrac12(|h(Y)|-1)$ there gives $|P(Y)|>\tfrac12(|h(Y)|+1)>\max_C|P|$.

Let $F=P-P(Y)$, pulled back to $\PP^n\setminus Q$. It is holomorphic and non-vanishing near $\R\PP^n$, and $F(y)=0$. The device again applies with $\Omega'=\Omega\setminus Q$.
\end{proof}

\begin{remark}
Proposition~\ref{prop:RPn} and Proposition~\ref{prop:exceptional} share a mechanism. In both cases the Matsuki orbit is a compact totally real orbit $K_0/(K_0)_y$ of maximal dimension in an \emph{affine} $K$-orbit $\kappa=K/K_y$; this is the complexification of $\R\PP^n$, respectively $\U(p)$. In both cases it is its own hull, cut out by regular functions on $\kappa$ with poles on the $K$-stable divisor. The difference between the two examples lies elsewhere:
\begin{itemize}
\item for $\SU(p,q)$, non-constant germs along lower-dimensional orbits require the Harish-Chandra chart, and then produce the larger hulls $\overline\cD\cap V_c$;
\item for $\SL_{n+1}(\R)$ on $\PP^n$, they come from the quadric $Q$, which is not the boundary of a bounded symmetric domain.
\end{itemize}
\end{remark}

%=====================================================================
\section{Proof of the main theorem}\label{sec:main}
%=====================================================================

\subsection{Orbits not of holomorphic type}

\begin{lemma}[Schubert codimension]\label{lem:schubert}
Let $F\subset\C^n$ have dimension $f$ and let $\max(0,k+f-n)<j\le\min(k,f)$. Then $\Sigma_j(F)=\{W\in\Gr_k(\C^n):\dim W\cap F\ge j\}$ is an irreducible subvariety of codimension $j(n-k-f+j)$.
\end{lemma}

\begin{proof}
The incidence variety $\{(L,W):L\in\Gr_j(F),\ L\subset W\}$ is a $\Gr_{k-j}(\C^n/L)$-bundle over $\Gr_j(F)$, of dimension $j(f-j)+(k-j)(n-k)$. Its projection onto $\Sigma_j(F)$ is birational, since a general $W\in\Sigma_j(F)$ meets $F$ in exactly $j$ dimensions. So $\codim\Sigma_j(F)=k(n-k)-j(f-j)-(k-j)(n-k)=j(n-k-f+j)$; see also \cite[\S9.4]{Fulton97}.
\end{proof}

\begin{theorem}[Theorem~\ref{thmD}(i)]\label{thm:vanishing}
If $M=M(a,b,c)$ is not of holomorphic type, then $\cO(U)=\C$ for every connected neighbourhood $U$ of $M$ in $Z$. This includes $c=0$, i.e.\ base cycles of flag domains not of holomorphic type.
\end{theorem}

\begin{proof}
By Lemma~\ref{lem:closure},
\[
Y_{\kappa(a,b)}=\Sigma_{a+1}(\Ep)\cup\Sigma_{b+1}(\Em).
\]
Every $W$ satisfies $\dim W\cap\Ep\ge k-q$, so $a\ge k-q$, and the range condition of Lemma~\ref{lem:schubert} holds whenever $\Sigma_{a+1}(\Ep)\ne\emptyset$. Its codimension is $(a+1)(q-k+a+1)$, where $q-k+a+1\ge1$. This equals $1$ exactly when $a=0$ and $k=q$. Likewise $\codim\Sigma_{b+1}(\Em)=(b+1)(p-k+b+1)$ equals $1$ exactly when $b=0$ and $k=p$; in that case $\Sigma_1(\Em)\ne\emptyset$. Hence $\codim Y_\kappa\ge2$ unless $M$ is of holomorphic type. Stabilisers are connected by Lemma~\ref{lem:connected}, so Proposition~\ref{thm:codim} applies.
\end{proof}

\begin{corollary}\label{cor:Kdivisors}
The $K$-stable prime divisors of $\Gr_k(\C^{p+q})$ are $\Sigma_1(\Em)=\{W\cap\Em\ne0\}$ if $k=p$ and $\Sigma_1(\Ep)=\{W\cap\Ep\ne0\}$ if $k=q$; there are no others. A $K$-orbit $\kappa(a,b)$ is disjoint from one of them if and only if $M(a,b,c)$ is of holomorphic type. Thus Theorem~\ref{thm:vanishing} and Theorem~\ref{thmD}(ii),(iii) are the two cases of Theorem~\ref{thmC} for $\SU(p,q)$.
\end{corollary}

\begin{proof}
A $K$-stable prime divisor does not contain the open orbit $\kappa_0$, so it is a codimension-one component of $Y_{\kappa_0}$. Let $(a_0,b_0,c_0)$ be the index of the closed $G_0$-orbit, which is dual to $\kappa_0$. By the computation above, the codimension-one components of $Y_{\kappa_0}=\Sigma_{a_0+1}(\Ep)\cup\Sigma_{b_0+1}(\Em)$ are $\Sigma_1(\Ep)$ when $a_0=0$ and $k=q$, and $\Sigma_1(\Em)$ when $b_0=0$ and $k=p$. The closed orbit has the maximal value $c_0=\min(p,q,k)$. If $k=q$, the constraint $a_0+c_0\le p$ forces $a_0=0$: either $c_0=q=k$, or $c_0=p$. Similarly, $b_0=0$ if $k=p$. The last statement follows from Lemma~\ref{lem:divisors} and the codimension computation in the proof of Theorem~\ref{thm:vanishing}.
\end{proof}

\subsection{Orbits of holomorphic type: the dictionary}

Assume $k=p$ and $b=0$, so $a=p-c$. In the chart $\cA=\{W\cap\Em=0\}$ we write $W=\Gamma_\Phi=\{v+\Phi v:v\in\Ep\}$ with $\Phi\in\Hom(\Ep,\Em)$. The group $K$ acts linearly by $(g_+,g_-)\cdot\Phi=g_-\Phi g_+^{-1}$. Let $\sigma_1\ge\dots\ge\sigma_p\ge0$ be the singular values of $\Phi$, so that $\|\Phi\|=\sigma_1$. Put
\[
R_c=\{\rank\Phi=c\},\qquad V_c=\overline{R_c}=\{\rank\Phi\le c\},\qquad K_c=\overline\cD\cap V_c=\{\Phi\in V_c:\|\Phi\|\le1\}.
\]

\begin{lemma}\label{lem:dictionary}
\begin{enumerate}[label=\textup{(\roman*)}]
\item $h(v+\Phi v)=\la(1-\Phi^*\Phi)v,v\ra$, so $\sig\Gamma_\Phi=(\#\{\sigma_i<1\},\#\{\sigma_i>1\},\#\{\sigma_i=1\})$, with all $p$ singular values counted.
\item $\Gamma_\Phi\cap\Ep=\ker\Phi$ and $\Gamma_\Phi\cap\Em=0$.
\item In the chart $\cA$:
\begin{itemize}
\item $O(p,0,0)=\cD=\{\|\Phi\|<1\}$;
\item $\kappa(p-c,0)=R_c$;
\item $M(p-c,0,c)=\{\Phi:\Phi^*\Phi\ \text{is an orthogonal projection of rank }c\}$, the manifold of rank-$c$ partial isometries;
\item $\Phi=0$, i.e.\ $W=\Ep$, is the base cycle of $\cD$.
\end{itemize}
\item The open $K$-orbit of $Z=\Gr_p(\C^n)$ is $R_{\min(p,q)}\subset\cA$, and $\kappa(a,b)\subset\overline{\kappa(r,s)}$ with $s=0$ holds exactly for $R_{c'}$ with $c'\ge c$.
\end{enumerate}
\end{lemma}

\begin{proof}
(i) and (ii) are direct, since $\Ep\perp_h\Em$. For (iii), $\sig\Gamma_\Phi=(p-c,0,c)$ and $\rank\Phi=c$ together force the $c$ nonzero singular values to equal $1$. For (iv), the open $K$-orbit is dual to the closed $G_0$-orbit, which has maximal $c$; and $s\le b=0$ in Lemma~\ref{lem:closure}.
\end{proof}

\begin{lemma}[Maximum principle on partial isometries]\label{lem:maxprinciple}
Let $P$ be a polynomial on $\Hom(\Ep,\Em)$ and $0\le\eta$. If $\Phi\in K_c$ and $\|\zeta\|\le\eta$, then
\[
|P(\Phi+\zeta)|\le\max\{|P(\Psi+\zeta)|:\Psi\in M\}.
\]
In particular $\max_{K_c}|P|=\max_M|P|$, and the polynomial hull of the closed $\eta$-neighbourhood $\overline{N_\eta(M)}$ contains $K_c+\overline{B_\eta}$.
\end{lemma}

\begin{proof}
Write $\Phi=\sum_{i=1}^c\sigma_i\,v_iu_i^*$ with orthonormal $(u_i)\subset\Ep$, $(v_i)\subset\Em$ and $\sigma_i\in[0,1]$. Let $\iota(z)=\sum_iz_i\,v_iu_i^*$ for $z\in\C^c$. Then $\iota(z)\in M$ for $|z_1|=\dots=|z_c|=1$. The polynomial $z\mapsto P(\iota(z)+\zeta)$ attains its maximum over the closed unit polydisc on the torus, and $\Phi=\iota(\sigma)$.
\end{proof}

\begin{lemma}\label{lem:polynomial}
Assume $c\ge1$ and $(c,p,q)\ne(p,p,p)$. Let $U\subset\cA$ be a connected $K_0$-invariant neighbourhood of $M$ and $f\in\cO(U)$ $K_0$-finite. Then $f$ is the restriction of a polynomial on $\Hom(\Ep,\Em)$.
\end{lemma}

\begin{proof}
Let $\kappa_0=R_{\min(p,q)}$ be the open $K$-orbit. As in the proof of Proposition~\ref{thm:codim}, $f=G$ on $U\cap\kappa_0$ for some $G\in\C[\kappa_0]$, and $G$ extends holomorphically to $KU$. By Lemma~\ref{lem:dictionary}(iv), $KU\supset\bigcup_{c'\ge c}R_{c'}=\Hom(\Ep,\Em)\setminus V_{c-1}$.

If $p\ne q$, the complement of $\kappa_0$ in $\Hom(\Ep,\Em)$ has codimension $|p-q|+1\ge2$, so $G$ is a polynomial. If $p=q$ and $c<p$, then $G=P/\det^m$ with $P$ a polynomial and the fraction reduced. Since $G$ is holomorphic at the points of $R_{p-1}\subset KU$, which are generic points of the irreducible hypersurface $\{\det=0\}$, we get $m=0$. In both cases $f=G$ on $U$, because $U\cap\kappa_0$ is dense in $U$.
\end{proof}

\subsection{The hull}

\begin{theorem}[Theorem~\ref{thmD}(ii)]\label{thm:hull}
Let $k=p$, $b=0$, $c\ge1$ and $(c,p,q)\ne(p,p,p)$. Then $K_c=\overline{\cD}\cap V_c$ is a holomorphic hull of $M=M(p-c,0,c)$ in $Z$, and it is the polynomial hull of $M$ in $\cA$. Moreover, every function holomorphic near $M$ extends holomorphically to the open set $K_c+B_\rho\subset\cA$ for some $\rho>0$.
\end{theorem}

\begin{proof}
\emph{(H1).} Let $f$ be holomorphic on a neighbourhood $U_0$ of $M$. Since $M$ is compact, $\cA$ is open in $Z$, and $t\mapsto tM$ is continuous, there are $\varepsilon,\eta>0$ with $\bigcup_{1\le t\le1+\varepsilon}\overline{N_\eta(tM)}\subset U_0\cap\cA$. Let $U$ be the connected component, containing $M$, of the open $K_0$-invariant set $\{z:K_0z\subset U_0\cap\cA\}$. Then $U$ contains the connected $K_0$-invariant set $\bigcup_t\overline{N_\eta(tM)}$, and in particular
\[
L:=\overline{N_\eta((1+\varepsilon)M)}\subset U .
\]

By Lemma~\ref{lem:polynomial}, each $f_\delta$ is a polynomial $P_\delta$. The central circle $\{(1,e^{i\theta})\}\subset K_0$ acts on polynomials by $P\mapsto P(e^{-i\theta}\,\cdot)$, hence by a single character on each irreducible type. So $P_\delta$ is homogeneous, of some degree $e_\delta$.

By Corollary~\ref{cor:fourier}(i), and Lemma~\ref{lem:maxprinciple} applied to the polynomial $\Phi\mapsto P_\delta((1+\varepsilon)\Phi)$ and to $\zeta/(1+\varepsilon)$,
\[
\sup_{(1+\varepsilon)K_c+B_\eta}|P_\delta|\le\sup_L|P_\delta|\le d_\delta^2\,C,\qquad C:=\sup_L|f| .
\]
Put $s=(1+\varepsilon/2)/(1+\varepsilon)<1$. Since $K_c\subset tK_c$ for $t\ge1$, we have $K_c+B_{s\eta}\subset s\bigl((1+\varepsilon)K_c+B_\eta\bigr)$, and homogeneity gives
\[
\sup_{K_c+B_{s\eta}}|P_\delta|\le s^{e_\delta}d_\delta^2\,C .
\]
The types $\delta$ with $f_\delta\ne0$ and $e_\delta=e$ are constituents of the space of homogeneous polynomials of degree $e$ on $\Hom(\Ep,\Em)$. Hence $\sum_{e_\delta=e}d_\delta^2\le\binom{pq+e-1}{e}^2$, which grows polynomially in $e$. So $\sum_\delta P_\delta$ converges uniformly on the connected open set $K_c+B_{s\eta}$ to a holomorphic function $\tilde f$. By Corollary~\ref{cor:fourier}(iii), $\tilde f=f$ on $U\cap(K_c+B_{s\eta})$, which contains a neighbourhood of $M$.

\emph{(H2).} Let $y\in Z\setminus K_c$; we use the device described after Definition~\ref{def:hull}.
\begin{itemize}
\item If $y\in\cA\setminus V_c$, some $(c+1)$-minor $m$ has $m(y)\ne0$ and $m=0$ on $K_c$. Use $F=m-m(y)$, which is non-vanishing near $K_c$.
\item If $y\in V_c$ with $\|y\|>1$, choose unit vectors with $|\la yu,v\ra|=\|y\|$. Use $F=\la\Phi u,v\ra-\la yu,v\ra$; it is non-vanishing on a neighbourhood of $K_c$, where $|\la\Phi u,v\ra|<\|y\|$.
\item If $y\notin\cA$, write Plücker coordinates of the graph as minors of $(I\ \Phi)$. Then $p_0(y)=0$ for the coordinate $p_0=1$ on $\cA$, and $p_J(y)\ne0$ for some $J$. The function $F=p_J/p_0$ is a polynomial on $\cA$, holomorphic near $K_c$, and unbounded near $y$. Here $\Omega'=\Omega\cap\cA$ is the complement of the hypersurface $\{p_0=0\}$.
\end{itemize}
\end{proof}

\begin{proposition}[Theorem~\ref{thmD}(iii)]\label{prop:exceptional}
If $c=p=q$, then $M=\{\Phi\text{ unitary}\}\cong\U(p)$ is totally real, of real dimension $p^2=\dim_\C\cA$, and $M$ is its own holomorphic hull in $Z$.
\end{proposition}

\begin{proof}
(H1) is trivial. For (H2), let $y\in\cA\setminus M$.
\begin{itemize}
\item If $y$ is singular, use $F=\det$. It does not vanish near $M$, so the germ $\det(\Phi)^{-1}$ along $M$ cannot be continued to $y$.
\item If $y$ is invertible and $\|y\|>1$, use $\la\Phi u,v\ra-\la yu,v\ra$ as before.
\item If $\|y^{-1}\|>1$, use $\la\Phi^{-1}v,u\ra-\la y^{-1}v,u\ra$ with $|\la y^{-1}v,u\ra|=\|y^{-1}\|$; it is holomorphic near $M$ and non-vanishing there, since $\|\Phi^{-1}\|=1$ on $M$.
\item If $y\notin\cA$, use the Plücker argument.
\end{itemize}
These cases are exhaustive: an invertible $y$ with $\|y\|\le1$ and $\|y^{-1}\|\le1$ has all singular values equal to $1$, so $y\in M$.
\end{proof}

The first case is the source of the discrepancy with the polynomial hull $\overline\cD$: in one variable, $z^{-1}$ is holomorphic near the unit circle but not at $0$, and $\det(\Phi)^{-1}$ plays the same role here. Remark~\ref{rem:det} explains this in detail.

\begin{remark}[Why $\det^{-1}$ breaks the polynomial picture]\label{rem:det}
For $c<\min(p,q)$, or $p\ne q$, the proof of Theorem~\ref{thm:hull} rests on Lemma~\ref{lem:polynomial}: every $K_0$-finite germ is a polynomial in $\Phi$. Write $V_{c'}$ for the matrices of rank $\le c'$. The reason the lemma holds is one of codimension. The orbit $\kappa=R_c$ is the set of matrices of rank exactly $c$, its closure is $V_c$, and its boundary $V_{c-1}$ has codimension
\[
(p-c+1)(q-c+1)-(p-c)(q-c)=p+q-2c+1
\]
in $V_c$. This is $\ge2$ unless $c=p=q$, so regular functions on $\kappa$ extend to the normal variety $V_c$ and are restrictions of polynomials (Lemma~\ref{lem:regpoly} gives a direct proof).

When $c=p=q$ the boundary is the divisor $\{\det\Phi=0\}$, and $\C[\kappa]=\C[\GL_p]=\C[\Phi][\det^{-1}]$. Already for $p=1$ the picture is familiar. Here $M$ is the unit circle and $\kappa=\C^*$. The function $z^{-1}$ is holomorphic near the circle but not at $0$, so the circle is its own hull in $\PP^1$, although its polynomial hull is the closed disc. In general $\det(\Phi)^{-1}$ is holomorphic near $\U(p)$ and singular along $\{\det=0\}$. This divisor meets $\overline\cD$ in the non-empty set of singular contractions, which contains the centre $0$. Hence no point of $\overline\cD\setminus\U(p)$ with $\det=0$ can be added to the hull. The functions $\la\Phi^{-1}v,u\ra$, which are again rational with poles on $\{\det=0\}$, exclude the remaining points. Polynomials alone cannot detect any of this, since for them $\max_{\U(p)}|P|=\max_{\overline\cD}|P|$ by Lemma~\ref{lem:maxprinciple}.
\end{remark}

\begin{remark}[Comparison with Kaup--Zaitsev]\label{rem:KZ}
$\Hom(\Ep,\Em)$ is the Cartan factor of type I, $K_0$ is its linear automorphism group, and $M$ is the $K_0$-orbit of a tripotent of rank $c$. For a tripotent $a$ of rank $c$ one has $\mu_j(z)\le\mu_j(a)$ for all $j$ iff $\sigma_1(z)\le1$ and $\rank z\le c$. So $K_c$ is the polynomial hull $Z(a)$ of \cite[Thm.~2.4]{KZ03}, and \cite[Thm.~2.5]{KZ03} gives continuous extension of continuous CR functions from $M$ to $K_c$. Theorem~\ref{thm:hull} adds two things: holomorphic extension of germs from neighbourhoods in the compact manifold $Z$ to a neighbourhood of $K_c$; and the identification of these orbits as precisely the Matsuki orbits of holomorphic type, together with the vanishing of Theorem~\ref{thm:vanishing} for all other Matsuki orbits. Theorem~\ref{thmF}(v) recovers the continuous extension from the representation-theoretic side, and identifies $CR^0(M)$ with $\cP(M)$.
\end{remark}

%=====================================================================
\section{Regular functions on \texorpdfstring{$K$}{K}-orbits}\label{sec:regular}
%=====================================================================

We use the following facts:
\begin{itemize}
\item regular functions on a normal variety extend across closed subsets of codimension $\ge2$;
\item the regular functions on the total space of a vector bundle $\mathcal E$ over $X$ are $\bigoplus_eH^0(X,S^e\mathcal E^*)$;
\item Cauchy's formula $S^e(V\otimes W)=\bigoplus_{|\lambda|=e}S_\lambda V\otimes S_\lambda W$ \cite[\S6]{FultonHarris}.
\end{itemize}

\begin{lemma}\label{lem:vanish}
Let $T\subset\underline V$ be a subbundle of a trivial bundle over a compact connected variety $X$ on which $\GL(V)$ acts compatibly. If $T_x\ne V$ for some $x$, then $H^0(X,S_\lambda T)=0$ for all partitions $\lambda\ne0$.
\end{lemma}

\begin{proof}
A section is a holomorphic map to the trivial bundle $S_\lambda V$, hence constant, with value in $\bigcap_xS_\lambda T_x$. This is a proper $\GL(V)$-submodule of the irreducible module $S_\lambda V$ (or of $0$), hence zero.
\end{proof}

\begin{proof}[Proof of Theorem~\ref{thmE}]
With the notation of Proposition~\ref{prop:K}, regular functions on $\kappa=\kappa(r,s)$ restrict to regular functions on the fibres $F$.

\emph{Case $d<m_1$, $d<m_2$.} Then $F\subset F'=\{S\cap V_2=0\}$, which is the total space of $\Hom(T,V_2)$ over $\Gr_d(V_1)$, with $T$ tautological. Its complement in $F'$ (the non-injective maps) has codimension $m_2-d+1\ge2$. So
\[
\C[F]=\bigoplus_\lambda H^0\bigl(\Gr_d(V_1),S_\lambda T\bigr)\otimes S_\lambda V_2^*=\C,
\]
by Lemma~\ref{lem:vanish}. Every $f\in\C[\kappa]$ is therefore constant on fibres, hence a holomorphic function on the compact connected base $B$, hence constant.

\emph{Case $d=m_1<m_2$.} Here $\kappa$ is the complement of a subset of codimension $\ge2$ in the total space of $\mathcal H=\Hom(Q_+,Q_-)$ over $B$, where $Q_+=\Ep/W_+$ and $Q_-=\Em/W_-$. Hence
\[
\C[\kappa]=\bigoplus_\lambda H^0\bigl(\Gr_r(\Ep),S_\lambda Q_+\bigr)\otimes H^0\bigl(\Gr_s(\Em),S_\lambda Q_-^*\bigr).
\]
If $s\ge1$, Lemma~\ref{lem:vanish} with $Q_-^*=W_-^0\subset\Em^*$ kills every $\lambda\ne0$, so $\C[\kappa]=\C$. If $s=0$, the second factor is $S_\lambda\Em^*$, and the first is nonzero for $\ell(\lambda)\le d$: the surjection $\Ep\otimes\cO\to Q_+$ gives a nonzero map $S_\lambda\Ep\to H^0(S_\lambda Q_+)$. So $\C[\kappa]$ is infinite-dimensional. Note that $s=0$ and $d=p-r$ mean exactly $s=0$ and $k=p$.

\emph{Case $d=m_2<m_1$.} Symmetric: $\C[\kappa]\ne\C$ iff $r=0$, i.e.\ iff $r=0$ and $k=q$.

\emph{Case $d=m_1=m_2$.} The fibre is $\mathrm{Iso}(V_1,V_2)$, and by the algebraic Peter--Weyl theorem
\[
\C[\kappa]=\bigoplus_{\lambda\in\mathbb Z^d\ \text{dominant}}H^0(\Gr_r,S_\lambda Q_+)\otimes H^0(\Gr_s,S_\lambda Q_-^*).
\]
By the Borel--Weil theorem on Grassmannians \cite[Ch.~4]{Weyman}, $H^0(\Gr_r,S_\lambda Q_+)=0$ if $r\ge1$ and $\lambda_d<0$, and $H^0(\Gr_s,S_\lambda Q_-^*)=0$ if $s\ge1$ and $\lambda_1>0$. So if $r,s\ge1$, only $\lambda=0$ survives. If $s=0$ (respectively $r=0$), all partitions (respectively all antidominant $\lambda$) contribute.

Finally, $d\ge1$ and $s=0$, $k=p$ mean $\kappa(r,0)=R_{p-r}$ with $r<p$, a rank stratum of $\cA$.
\end{proof}

\begin{corollary}\label{cor:kappa-germs}
If $c\ge1$ and $M$ is not of holomorphic type, then every function holomorphic on a neighbourhood of $M$ in $\kappa$ is constant.
\end{corollary}

\begin{proof}
By Theorem~\ref{thmE}, $\C[\kappa(a,b)]=\C$. The proofs of Theorem~\ref{thm:alg} and Corollary~\ref{cor:fourier} apply verbatim with $Z$ replaced by the $K$-manifold $\kappa$. Applying them to a connected $K_0$-invariant neighbourhood of $M$ in $\kappa$: for $\delta\ne1$, $f_\delta$ agrees near $M$ with an element of $\C[\kappa]_\delta=0$, so $f_\delta=0$; and $f=f_1$ agrees with a constant.
\end{proof}

%=====================================================================
\section{Global CR functions and Dolbeault cohomology in degree zero}\label{sec:CR}
%=====================================================================

Throughout this section $M=M(a,b,c)$ with $c\ge1$ and $\kappa=\kappa(a,b)$. By Theorem~\ref{thm:structure}(iii), $M$ is a compact, real-analytic, generic CR submanifold of the complex manifold $\kappa$, and $K_0$ acts on $M$ by CR automorphisms. In particular $M$ is a uniqueness set: a holomorphic function on a connected neighbourhood of $M$ in $\kappa$ that vanishes on $M$ vanishes identically. Theorem~\ref{thm:alg} and Corollary~\ref{cor:fourier} hold, with the same proofs, for the $K$-manifold $\kappa$ in place of $Z$; the stabilisers are connected by Lemma~\ref{lem:connected}.

\begin{lemma}\label{lem:CRextension}
Every $K_0$-finite $\varphi\in CR^\infty(M)$ is the restriction to $M$ of some $F\in\C[\kappa]$.
\end{lemma}

\begin{proof}
The span $V_\varphi$ of the $K_0$-translates of $\varphi$ is finite-dimensional. It consists of matrix coefficients on $M=K_0/L_0$, hence of real-analytic CR functions. A real-analytic CR function on a real-analytic generic submanifold extends holomorphically to a neighbourhood, uniquely because $M$ is a uniqueness set \cite[\S1.7]{BER99}.

Extend a basis $\varphi_1,\dots,\varphi_N$ of $V_\varphi$ to holomorphic functions $\Phi_i$ on a common neighbourhood $W$ of $M$ in $\kappa$. Let $U$ be the connected component, containing $M$, of $\{z:K_0z\subset W\}$. For $k\in K_0$, the functions $\Phi_i\circ k^{-1}$ and the extension of $\varphi_i\circ k^{-1}$ are holomorphic on $U$ and agree on $M$, hence agree on $U$. So the extension $\Phi$ of $\varphi$ is a $K_0$-finite holomorphic function on $U$. By Theorem~\ref{thm:alg}, applied in $\kappa$ at a point of $M$, there is $F\in\C[\kappa]$ with $F=\Phi$ near $M$, hence $F|_M=\varphi$.
\end{proof}

\begin{lemma}[Approximation]\label{lem:approx}
Every $\varphi\in CR^0(M)$ is a uniform limit on $M$ of $K_0$-finite elements of $CR^\infty(M)$.
\end{lemma}

\begin{proof}
Let $\varepsilon>0$. Choose a continuous approximate identity $\psi\ge0$ on $K_0$ with $\|\psi*\varphi-\varphi\|_\infty<\varepsilon$, where $(\psi*\varphi)(y)=\int_{K_0}\psi(k)\varphi(k^{-1}y)\,dk$, and, by the Peter--Weyl theorem, a $K_0$-finite function $\psi'$ with $\|\psi-\psi'\|_{L^1}\|\varphi\|_\infty<\varepsilon$. Then $\|\psi'*\varphi-\varphi\|_\infty<2\varepsilon$.

The function $\psi'*\varphi$ is $K_0$-finite, because $\varphi\mapsto\psi'*\varphi$ intertwines left translations with the left translates of $\psi'$, which span a finite-dimensional space. It is CR, because each translate $\varphi\circ k^{-1}$ is a CR distribution ($k$ is a biholomorphism of $Z$ preserving $M$), and averaging preserves the distributional CR equations. Being $K_0$-finite, it is a finite sum of matrix coefficients on $M=K_0/L_0$, hence real-analytic, so it lies in $CR^\infty(M)$.
\end{proof}

\begin{lemma}\label{lem:regpoly}
Let $k=p$, $b=0$, $c\ge1$ and $(c,p,q)\ne(p,p,p)$, so that $\kappa=R_c\subset\cA$. Every $F\in\C[R_c]$ is the restriction of a polynomial on $\Hom(\Ep,\Em)$.
\end{lemma}

\begin{proof}
Since $\C[R_c]$ is a locally finite $K$-module, $F$ is a finite sum $\sum_dF_d$ of eigenfunctions of the central torus $\{(1,t)\}\subset K$, which acts by $\Phi\mapsto t\Phi$; thus $F_d(t\Phi)=t^dF_d(\Phi)$. It suffices to treat one $F_d$.

\emph{Case $c<\min(p,q)$.} Let $\mathcal U\subset\Hom(\Ep,\C^c)\times\Hom(\C^c,\Em)$ be the set of pairs $(A,B)$ with $A$ surjective and $B$ injective, and $\mu(A,B)=BA$. Then $\mu:\mathcal U\to R_c$ is surjective, and $G=F_d\circ\mu$ is holomorphic on $\mathcal U$ and invariant under $(A,B)\mapsto(gA,Bg^{-1})$, $g\in\GL_c$. The complement of $\mathcal U$ has codimension $\min(p,q)-c+1\ge2$, so $G$ extends to an entire function \cite[Ch.~7]{GrauertRemmert}. It is homogeneous of degree $d$ in $A$ and in $B$ separately, hence $d\ge0$ and $G$ is a $\GL_c$-invariant polynomial. By the first fundamental theorem of invariant theory for $\GL_c$ \cite{Weyl46}, $G$ is a polynomial in the entries of $BA$, i.e.\ $G=P\circ\mu$ for a polynomial $P$ on $\Hom(\Ep,\Em)$. Thus $F_d=P$ on $R_c$.

\emph{Case $c=p<q$.} Then $R_p$ is the set of injective maps, whose complement in $\Hom(\Ep,\Em)$ has codimension $q-p+1\ge2$. So $F_d$ extends to an entire homogeneous function, which is a polynomial.

\emph{Case $c=q<p$.} Then $R_q$ is the set of surjective maps, with complement of codimension $p-q+1\ge2$, and the same argument applies. Since $c\le\min(p,q)$, these cases exhaust all possibilities except $c=p=q$.
\end{proof}

\begin{proof}[Proof of Theorem~\ref{thmF}]
(i) One inclusion is Lemma~\ref{lem:CRextension}. Conversely, the restriction to $M$ of $F\in\C[\kappa]$ is CR and $K_0$-finite, since $\C[\kappa]$ is a locally finite $K$-module.

(ii) Suppose $M$ is not of holomorphic type. Then $\C[\kappa]=\C$ by Theorem~\ref{thmE}. For $\varphi\in CR^\infty(M)$ and an irreducible type $\delta$, the isotypic component $\varphi_\delta=d_\delta\int_{K_0}\overline{\chi_\delta(k)}\,\varphi\circ k^{-1}\,dk$ is again CR, because $K_0$ acts by CR automorphisms. It is $K_0$-finite of type $\delta$, so by (i) it is the restriction of an element of $\C[\kappa]_\delta$, which is $0$ for $\delta\ne1$. Hence $\varphi$ is $K_0$-invariant, and constant because $M$ is a single $K_0$-orbit.

Now suppose $M$ is of holomorphic type. Then $\C[\kappa]$ contains a non-constant $F$ by Theorem~\ref{thmE}, applied with $(r,s)=(a,b)$ and $d=c\ge1$. Its restriction to $M$ is non-constant, since otherwise $F$ would be constant near $M$ and hence on $\kappa$.

(iii) Injectivity holds because $M$ is a uniqueness set in the connected manifold $\kappa$.

If $M$ is not of holomorphic type, let $f\in\cO(\kappa)$. By Corollary~\ref{cor:fourier} on $\kappa$, each $f_\delta$ is $K_0$-finite and holomorphic on $\kappa$, hence lies in $\C[\kappa]_\delta$ by Theorem~\ref{thm:alg} and the identity theorem. So $f_\delta=0$ for $\delta\ne1$, and $f$ is constant. Thus $\cO(\kappa)=\C=CR^\infty(M)$.

If $M$ is of holomorphic type with $k=p$ and $(c,p,q)\ne(p,p,p)$, work in the chart $\cA$, where $\kappa=R_c$ and $M$ is the set of rank-$c$ partial isometries. Choose orthonormal $u_1,\dots,u_c\in\Ep$ and $v_1,\dots,v_c\in\Em$, put $u=u_1$, $v=v_1$, and let $\Phi_0=2vu^*+\sum_{i\ge2}v_iu_i^*\in R_c$. Then $\varphi=1/(\la\Phi u,v\ra-2)$ is real-analytic and CR on all of $M$. Indeed, every $\Phi\in M$ is a partial isometry, so $\|\Phi\|=1$; since $u$ and $v$ are unit vectors, $|\la\Phi u,v\ra|\le1<2$, and the denominator is a holomorphic polynomial without zeros on $M$. If $\varphi=F|_M$ with $F\in\cO(\kappa)$, then $F\cdot(\la\Phi u,v\ra-2)-1$ vanishes on $M$, hence on $\kappa$, which is impossible at $\Phi_0$.

If $c=p=q$, then $M\cong\U(p)$ is totally real of maximal dimension in $\kappa=\mathrm{Iso}(\Ep,\Em)$. Every smooth function on $M$ is CR, while restrictions of holomorphic functions are real-analytic. The case $a=0$, $k=q$ is symmetric.

(iv) If $M$ is not of holomorphic type, every $\varphi\in CR^0(M)$ is a uniform limit of elements of $CR^\infty(M)=\C$ by Lemma~\ref{lem:approx} and (ii), hence constant. If $M$ is of holomorphic type, $CR^0(M)\supset CR^\infty(M)\ne\C$ by (ii).

(v) Since $M\subset\cA$, restrictions of polynomials in $\Phi$ are CR, and $CR^0(M)$ is uniformly closed; so $\cP(M)\subset CR^0(M)$. Conversely, by Lemma~\ref{lem:approx}, (i) and Lemma~\ref{lem:regpoly}, every $\varphi\in CR^0(M)$ is a uniform limit of restrictions of polynomials. Hence $CR^0(M)=\cP(M)$.

Here the polynomial hull is taken in the vector space $\Hom(\Ep,\Em)\cong\cA$, not in $Z$; it coincides with the hull of Theorem~\ref{thm:hull}, a compact subset of $Z$, only because both equal the compact set $K_c$, which lies in $\cA$. The polynomially convex hull of $M$ in $\Hom(\Ep,\Em)$ is $K_c$. Indeed, the $(c+1)$-minors are polynomials vanishing on $M$, so the hull lies in $V_c$. The closed operator-norm ball is convex and compact, hence polynomially convex, so the hull lies in $\overline\cD$. Conversely $K_c$ lies in the hull by Lemma~\ref{lem:maxprinciple} with $\eta=0$. For a compact set $S\subset\C^N$ the maximal ideal space of $\cP(S)$ is the polynomial hull $\widehat S$, and restriction $\cP(\widehat S)\to\cP(S)$ is an isometric isomorphism by the maximum principle \cite[Ch.~II--III]{Gamelin69}. This proves the claim.

If $c=p=q$, the CR dimension of $M$ is $0$, so every continuous function is CR in the sense of distributions and $CR^0(M)=C(M)$.
\end{proof}

\begin{remark}
Marini and Nacinovich \cite{MN18} prove isomorphisms between tangential Cauchy--Riemann cohomology of $\mathfrak n$-reductive compact homogeneous CR manifolds and Dolbeault cohomology of their canonical embeddings, under pseudoconcavity assumptions. Theorem~\ref{thmF}(iii) shows that in degree zero, for Matsuki orbits of $\SU(p,q)$ in Grassmannians, the isomorphism holds for \emph{every} orbit not of holomorphic type, with no pseudoconcavity hypothesis. For instance it holds for $M(1,1,1)\subset\Gr_3(\C^8)$, whose Levi form is degenerate (Proposition~\ref{prop:levi}). It fails exactly for the orbits of holomorphic type; for $c=1$ these are strictly pseudoconvex.
\end{remark}

\subsection{A geometric proof of constancy by rational curves}

Theorem~\ref{thmF}(ii) rests on Theorem~\ref{thmE}, i.e.\ on the algebra of regular functions on $\kappa$. In the range $c<\min(p-a,q-b)$ there is a direct geometric reason for constancy: $M$ contains many compact complex curves, and they connect any two points of a fibre.

\begin{lemma}\label{lem:isolines}
Let $Q$ be a Hermitian space of signature $(s,t)$ with $s,t\ge2$, and let $x,y\in Q$ be non-zero isotropic vectors. Then there is a non-zero isotropic $z\in Q$ such that $\la x,z\ra_\C$ and $\la z,y\ra_\C$ are totally isotropic \textup{(}possibly $z\in\{x,y\}$\textup{)}.
\end{lemma}

\begin{proof}
If $h(x,y)=0$ take $z=y$. Otherwise $\la x,y\ra_\C$ is non-degenerate of signature $(1,1)$, so its $h$-orthogonal complement has signature $(s-1,t-1)$ with both entries $\ge1$, and contains a non-zero isotropic $z$. Then $z\perp_hx$ and $z\perp_hy$.
\end{proof}

\begin{lemma}\label{lem:chain}
Let $V'$ be a Hermitian space and $N_1,N_2\in\Iso_c(V')$, $c\ge1$. There is a sequence $N_1=X_0,X_1,\dots,X_m=N_2$ in $\Iso_c(V')$ with $\dim(X_{i-1}\cap X_i)\ge c-1$ for all $i$.
\end{lemma}

\begin{proof}
Induction on $e=c-\dim(N_1\cap N_2)$; there is nothing to prove if $e=0$. If $e>0$, put $Y=N_1\cap N_2$ and choose $w\in N_2\setminus N_1$. Since $N_2$ is totally isotropic, $Y\subset N_1\cap w^{\perp_h}$, and $\dim(N_1\cap w^{\perp_h})\ge c-1\ge\dim Y$. Choose $Y\subset S\subset N_1\cap w^{\perp_h}$ with $\dim S=c-1$ and put $X_1=S\oplus\C w$. It is totally isotropic and $c$-dimensional, $\dim(N_1\cap X_1)\ge c-1$, and $X_1\cap N_2\supset Y\oplus\C w$, so $c-\dim(X_1\cap N_2)<e$.
\end{proof}

\begin{proposition}\label{prop:curves}
Let $c\ge1$ and $c<\min(p-a,q-b)$. Through any two points of a fibre of $\pi|_M:M\to\Gr_a(\Ep)\times\Gr_b(\Em)$ passes a connected chain of rational curves $\PP^1\subset M$. Consequently every $C^1$ CR function on $M$ is constant.
\end{proposition}

\begin{proof}
Write $B=\Gr_a(\Ep)\times\Gr_b(\Em)$, fix $\beta=(W_+,W_-)\in B$, and put $V'=(W_+\oplus W_-)^{\perp_h}$, of signature $(p',q')=(p-a,q-b)$. By Theorem~\ref{thm:structure}, the fibre is $F_\beta=\{W_+\oplus W_-\oplus N:N\in\Iso_c(V')\}$.

\emph{Rational curves.} Let $S\subset L\subset V'$ be totally isotropic of dimensions $c-1$ and $c+1$; such $L$ exist because $c+1\le\min(p',q')$. Then
\[
C_{S,L}=\{W_+\oplus W_-\oplus N:\ S\subset N\subset L\}\cong\PP(L/S)\cong\PP^1
\]
is a compact complex curve in $Z$ contained in $F_\beta\subset M$.

\emph{Chains.} For fixed $S$, the planes $N\in\Iso_c(V')$ containing $S$ correspond to the isotropic lines of $S^{\perp_h}/S$, a Hermitian space of signature $(p'-c+1,q'-c+1)$ with both entries $\ge2$. By Lemma~\ref{lem:isolines}, any two of them are joined by at most two curves $C_{S,L}$. Two isotropic $c$-planes meeting in dimension $\ge c-1$ contain a common $S$, so by Lemma~\ref{lem:chain} any two points of $F_\beta$ are joined by a chain of such curves.

\emph{Constancy.} Let $f$ be a $C^1$ CR function on $M$. For a complex submanifold $\Sigma\subset M$ of $Z$ we have $T^{0,1}\Sigma\subset\C TM\cap T^{0,1}Z=T^{0,1}M$, so $f|_\Sigma$ is holomorphic. Hence $f$ is constant on each $C_{S,L}$, and therefore on each fibre $F_\beta$. So $f=g\circ\pi$ with $g$ a $C^1$ function on $B$ (compose with local smooth sections).

We show that $\pi_*T^{0,1}_WM=T^{0,1}_\beta B$ for $W\in F_\beta$. The fibre $F_\beta$ is the Matsuki orbit $M(0,0,c)$ of $\SU(V')$ in $\Gr_c(V')$, whose dual $K$-orbit is the fibre $\pi^{-1}(\beta)\subset\kappa$ of Proposition~\ref{prop:K}(ii) (identify $E_\pm/W_\pm$ with $E_\pm\ominus W_\pm$). By Theorem~\ref{thm:structure}(iii), applied to $\SU(V')$, $F_\beta$ is generic of real codimension $c^2$ in $\pi^{-1}(\beta)$. Hence the kernel of $\pi_*$ on $T^{0,1}_WM$, which is $T^{0,1}_WM\cap\C T_WF_\beta=T^{0,1}_WF_\beta$, has dimension $\dim_\C\pi^{-1}(\beta)-c^2$. The CR dimension of $M$ is $\dim_\C\kappa-c^2$. So $\pi_*T^{0,1}_WM$ has dimension $\dim_\C\kappa-\dim_\C\pi^{-1}(\beta)=\dim_\C B$, and it lies in $T^{0,1}_\beta B$ because $\pi$ is holomorphic.

Thus for $\bar Z\in T^{0,1}_\beta B$ there is $\bar X\in T^{0,1}_WM$ with $\pi_*\bar X=\bar Z$, and $\bar Zg=\bar Xf=0$. So $g$ is holomorphic on the compact connected manifold $B$, hence constant.
\end{proof}

\begin{remark}
If $c<\min(p-a,q-b)$, then $M$ is not of holomorphic type: $k=q$ and $a=0$ would force $c=q-b$. So Proposition~\ref{prop:curves} is a special case of Theorem~\ref{thmF}(ii),(iv), with a proof that avoids representation theory. It does not cover all orbits that are not of holomorphic type. For instance $M(1,1,1)\subset\Gr_3(\C^8)$ has $c=1=p-a$, its fibres are strictly pseudoconvex spheres $S^9$ containing no complex curves, and constancy there genuinely needs Theorem~\ref{thmE} (Example~\ref{ex:ballbundle}).
\end{remark}

%=====================================================================
\section{Sections of line bundles}\label{sec:bundles}
%=====================================================================

Every line bundle on $Z=\Gr_k(\C^n)$ is a power $\cO_Z(m)$ of the Plücker bundle. Let $\widehat Z\subset\Lambda^k\C^n$ be the affine cone over the Plücker embedding and $\pi:\widehat Z\setminus\{0\}\to Z$ the projection, so that $\widehat Z\setminus\{0\}$ is the complement of the zero section in $\cO_Z(-1)$. Holomorphic sections of $\cO_Z(m)$ over $U\subset Z$ correspond to holomorphic functions $f$ on $\pi^{-1}(U)$ with $f(tw)=t^mf(w)$ for $t\in\C^*$.

The homogeneous coordinate ring of the Plücker embedding is $\C[\widehat Z]=\bigoplus_{m\ge0}H^0(Z,\cO_Z(m))$, and $\widehat Z$ is normal \cite[Ch.~9]{Fulton97}, \cite{Brion05}. Let $\widehat K=K\times\C^*$ act on $\widehat Z$ by $(g,t)\cdot w=t\,\Lambda^k(g)w$, with maximal compact subgroup $\widehat K_0=K_0\times\U(1)$. The $\widehat K$-orbits in $\widehat Z\setminus\{0\}$ are the sets $\pi^{-1}(\kappa')$, $\kappa'$ a $K$-orbit. The stabiliser of $w\in\pi^{-1}(W)$ is
\[
\{(g,t):\ g\in K_W,\ t\det(g|_W)=1\}\cong K_W ,
\]
which is connected by Lemma~\ref{lem:connected}.

\begin{proof}[Proof of Theorem~\ref{thmG}]
Uniqueness holds because $U$ is connected. For existence, it suffices to treat a connected $K_0$-invariant neighbourhood $U'\subset U$ of $M$: a global section agreeing with $s$ on $U'$ agrees with it on $U$, by the identity theorem.

Let $f$ be the homogeneous function of degree $m$ on $\widehat U=\pi^{-1}(U')$ corresponding to $s$. The set $\widehat U$ is connected, $\widehat K_0$-invariant and $\C^*$-stable. Decompose $f=\sum_\delta f_\delta$ into $\widehat K_0$-types; each $f_\delta$ is again homogeneous of degree $m$. Repeat the proof of Proposition~\ref{thm:codim} on the $\widehat K$-manifold $\widehat Z\setminus\{0\}$, using the open orbit $\pi^{-1}(\kappa_0)$ and the connected set $\widehat U\cap\pi^{-1}(\kappa_0)$. This shows that $f_\delta$ coincides on $\widehat U\cap\pi^{-1}(\kappa_0)$ with a regular function on $\pi^{-1}(\kappa_0)$, which extends holomorphically to $\widehat K\widehat U=\pi^{-1}(KU')\supset\pi^{-1}(Z\setminus Y_\kappa)$.

As $M$ is not of holomorphic type, the proof of Theorem~\ref{thm:vanishing} gives $\codim Y_\kappa\ge2$, so $\codim\pi^{-1}(Y_\kappa)\ge2$ in the manifold $\widehat Z\setminus\{0\}$. The extension therefore extends to $\widehat Z\setminus\{0\}$, by the Riemann extension theorem across analytic subsets of codimension $\ge2$ in a manifold. It then extends across the vertex, for the following reasons:
\begin{itemize}
\item the vertex is a single point of the normal complex space $\widehat Z$;
\item $\dim\widehat Z=k(n-k)+1\ge2$, so the vertex has codimension $\ge2$;
\item holomorphic functions on a normal complex space extend across analytic subsets of codimension $\ge2$ (second Riemann extension theorem \cite[Ch.~7]{GrauertRemmert}).
\end{itemize}
This gives a holomorphic function $G_\delta$ on the affine variety $\widehat Z$, homogeneous of degree $m$.

Since $\widehat Z$ is a closed analytic subset of the Stein manifold $\Lambda^k\C^n$, Cartan's Theorem~B applied to its ideal sheaf shows that the restriction $\cO(\Lambda^k\C^n)\to\cO(\widehat Z)$ is surjective. So $G_\delta$ extends to an entire function on $\Lambda^k\C^n$. Its homogeneous Taylor components of degree $\ne m$ vanish on the cone $\widehat Z$. Hence $G_\delta=0$ if $m<0$, and otherwise $G_\delta\in\C[\widehat Z]_m=H^0(Z,\cO_Z(m))$. Only the finitely many types occurring in $H^0(Z,\cO_Z(m))$ can have $f_\delta\ne0$. So $f$ is the finite sum $\sum_\delta G_\delta$, by Corollary~\ref{cor:fourier}, and $s$ is the restriction of the corresponding global section.
\end{proof}

\begin{remark}
For $M$ of holomorphic type with $k=p$, the Plücker coordinate $p_0$ does not vanish on the chart $\cA$. So $s\mapsto s/p_0^m$ identifies sections of $\cO_Z(m)$ near $M$ with holomorphic functions near $M$, and Theorem~\ref{thm:hull} and Proposition~\ref{prop:exceptional} hold verbatim for sections of $\cO_Z(m)$.
\end{remark}

%=====================================================================
\section{Comparison with cycle spaces}\label{sec:cycles}
%=====================================================================

For a Matsuki pair $(\kappa,\gamma)$ let $\mathcal C(\gamma)$ be the connected component of the identity of $\{g\in G:\ g\kappa\cap\gamma\ \text{is non-empty and compact}\}$ \cite{GM03,HN03}. Gindikin and Matsuki \cite{GM03} call two particular closed $K$-orbits of a Hermitian real form \emph{of holomorphic type}; every non-closed $K$-orbit is of non-holomorphic type. They conjectured that $\mathcal C(\gamma)=\Om$ whenever $\kappa$ is of non-holomorphic type. This was proved in the remaining cases by Matsuki \cite{MatsukiIII}; in particular $\mathcal C(\gamma)=\Om$ whenever $\kappa$ is not closed, see also \cite{HN03,Ntatin}.

\begin{proposition}\label{prop:cyclecompact}
Let $g\in G$ with $P:=g\Ep$ $h$-positive definite and $Q:=g\Em$ $h$-negative definite. Then for every admissible $(a,b,c)$,
\[
\begin{aligned}
g\kappa(a,b)\cap O(a,b,c)=\bigl\{W_+\oplus W_-\oplus N:\ &W_+\in\Gr_a(P),\ W_-\in\Gr_b(Q),\\ &N\in\Iso_c\bigl((W_+\oplus W_-)^{\perp_h}\bigr)\bigr\},
\end{aligned}
\]
which is non-empty and compact.
\end{proposition}

\begin{proof}
Note that $g\kappa(a,b)=\{W:\dim W\cap P=a,\ \dim W\cap Q=b\}$, but $P$ and $Q$ need not be $h$-orthogonal. Denote the right-hand side by $T$.

\emph{$T$ is compact and non-empty.} For $W_+\subset P$ and $W_-\subset Q$, the space $W_+\oplus W_-$ contains an $a$-dimensional positive and a $b$-dimensional negative subspace, so it is non-degenerate of signature $(a,b)$. Its $h$-orthogonal complement has signature $(p-a,q-b)$, and $\Iso_c$ of it is non-empty and compact by admissibility and Lemma~\ref{lem:Iso}. Hence $T$, the total space of a fibre bundle with compact fibres over $\Gr_a(P)\times\Gr_b(Q)$, is compact and non-empty.

\emph{$T$ lies in $g\kappa(a,b)\cap O(a,b,c)$.} Each $W\in T$ has signature $(a,b,c)$, since $N$ is isotropic and orthogonal to $W_+\oplus W_-$. If $v=w_++w_-+\nu\in W\cap P$, then $u:=w_-+\nu=v-w_+\in P$, so $h(u)\ge0$. On the other hand $h(u)=h(w_-)\le0$, since $\nu\perp_hw_-$ and $\nu$ is isotropic. So $h(u)=0$, hence $u=0$ because $P$ is definite. Thus $W\cap P=W_+$, and symmetrically $W\cap Q=W_-$.

\emph{Conversely.} Let $W\in g\kappa(a,b)\cap O(a,b,c)$ and put $W_+=W\cap P$, $W_-=W\cap Q$, $N=N_W$. Suppose $v=w_++w_-\in N$. Then $h(v,w_+)=0=h(v,w_-)$, which gives $h(w_+)=-\overline{\alpha}$ and $h(w_-)=-\alpha$ with $\alpha=h(w_+,w_-)$. Since $h(w_+)$ is real, $\alpha$ is real. If $w_+\ne0$ then $\alpha<0$, and if $w_-\ne0$ then $\alpha>0$. Hence $w_+=w_-=0$, the sum $W_+\oplus W_-\oplus N$ is direct, and by dimension it equals $W$.
\end{proof}

\begin{proof}[Proof of Theorem~\ref{thmH}]
(i) By Proposition~\ref{prop:cyclecompact}, the connected set $\{g:\ g\Ep\in\cD_+,\ g\Em\in\cD_-\}$, which contains $e$, lies in $\{g:\ g\kappa\cap\gamma\ \text{non-empty and compact}\}$, hence in $\mathcal C(\gamma)$. Its image in $G/K$ is $\cD_+\times\cD_-$, the Akhiezer--Gindikin domain of $\SU(p,q)$ in its Hermitian realisation \cite{BHH03}, \cite[Part~II]{FHW}.

(ii) Let $k\le p$. The closed $K$-orbit $\kappa(k,0)=\Gr_k(\Ep)$ is the base cycle of the flag domain $O(k,0,0)$. Since $\Gr_k(g\Ep)$ is compact and connected, $g\Gr_k(\Ep)\cap O(k,0,0)$ is compact and non-empty iff $\Gr_k(g\Ep)\subset O(k,0,0)$. Every vector of $g\Ep$ lies in a $k$-dimensional subspace, so this holds iff $g\Ep$ is positive definite. Hence $\mathcal C(O(k,0,0))=\{g:\ g\Ep\in\cD_+\}^0$, and the cycles $g\Gr_k(\Ep)=\Gr_k(g\Ep)$ are parametrised by $\cD_+\cong\cD$. This is strictly larger than $\Om$, so by the theorem of Matsuki \cite{MatsukiIII} $\kappa(k,0)$ is of holomorphic type. The case $\kappa(0,k)$, $k\le q$, is symmetric.

(iii) If $M(p-c,0,c)$ is of holomorphic type with $c\ge1$, then $\kappa=R_c$ is not closed, since $0\in\overline{R_c}$. So $\kappa$ is of non-holomorphic type and $\mathcal C(\gamma)=\Om$ by \cite{MatsukiIII}. For $c=0$, the orbit $\kappa=\{\Ep\}=\kappa(p,0)$ is of holomorphic type by (ii). The symmetric case is the same. The two notions therefore agree exactly at the base points $\{\Ep\}$ and $\{\Em\}$.
\end{proof}

\begin{remark}
Theorem~\ref{thmH} separates two phenomena that both originate in the Harish-Chandra chart. Gindikin--Matsuki exceptionality concerns the family of moved cycles $g\kappa\cap\gamma$; it occurs only for closed $K$-orbits, whose Matsuki orbits are compact complex. Function-theoretic non-triviality concerns $M$ itself; it occurs for every rank stratum $R_c$, and its hull $\overline{\cD}\cap V_c$ contains the base point $\Ep$, the unique point where the two notions meet.
\end{remark}

%=====================================================================
\section{Levi forms in the case \texorpdfstring{$c=1$, $a=p-1$}{c=1, a=p-1}}\label{sec:levi}
%=====================================================================

Let $c=1$ and $a=p-1$, so $k=p+b$. Then $d=1=p-a$, and Proposition~\ref{prop:K} identifies $\kappa(p-1,b)$ with $L\setminus0_B$, where
\[
L=\Hom(\Ep/W_+,\Em/W_-)\cong\cO(-1)\boxtimes Q_b\qquad\text{over }B=\PP^{p-1}\times\Gr_b(\Em).
\]
Give $L$ the $K_0$-invariant metric obtained by identifying $E_\pm/W_\pm$ with $W_\pm^\perp\cap E_\pm$. Writing $W=W_+\oplus W_-\oplus\{u+Au\}$ with $u\in W_+^\perp\cap\Ep$, the Gram matrix of $h|_W$ is $\operatorname{diag}(I_{p-1},-I_b,|u|^2(1-|A|^2))$. Hence
\[
M=\{|A|=1\},\qquad\kappa\cap O(p,b,0)=\{0<|A|<1\},\qquad\kappa\cap O(p-1,b+1,0)=\{|A|>1\}.
\]

\begin{proposition}\label{prop:levi}
With respect to $\psi=\log|A|^2$, the Levi form of $M$ in $\kappa$ has signature
\[
(p+q-b-2,\ b,\ b(q-b-1)),
\]
counting positive, negative and zero eigenvalues. Its kernel is the tangent space of the $\nu$-fibres of Theorem~\ref{thm:structure}(iv).
\end{proposition}

\begin{proof}
Split $H(M)$ into vertical and horizontal parts using the Chern connection. In a normal frame,
\[
i\partial\bar\partial\log|A|^2=\omega_{FS}^{\mathrm{vert}}-\la i\Theta_L(\cdot,\bar\cdot)A,A\ra/|A|^2,
\]
which is the computation behind Kobayashi's criterion for negativity of vector bundles \cite{Kobayashi75,Griffiths69}.
\begin{itemize}
\item The vertical part is positive on $A^\perp\subset L_x$, of rank $q-b-1$.
\item The curvature of $\cO(-1)$ contributes $+\omega_{FS}$ on $T\PP^{p-1}$, of rank $p-1$.
\item For the quotient bundle $Q_b$ of the trivial bundle $\Em$, $\la i\Theta_{Q_b}(\xi,\bar\xi)a,a\ra=\|\xi^*a\|^2$, where $a=A(u)$ for unit $u$. As $\xi$ ranges over $\Hom(W_-,\Em/W_-)$, the vector $\xi^*a$ fills $W_-$. This gives $b$ negative directions and a kernel $\{\xi:\xi(W_-)\perp a\}$ of dimension $b(q-b-1)$.
\end{itemize}
The $\nu$-fibres keep $W_+$ and $N=\{u+Au\}$ fixed and move $W_-$ inside $a^\perp\cap\Em$, so they are horizontal and tangent to that kernel.
\end{proof}

For $b=0$ the Levi form is positive definite and $M$ bounds the strictly pseudoconvex disc bundle $\{|A|<1\}$, in accordance with Theorem~\ref{thm:hull}. For $b\ge1$ there are negative eigenvalues; by Lewy's theorem \cite{Lewy56}, CR functions then extend locally to both sides of $M$, and Theorem~\ref{thm:vanishing} shows that all of them are constant.

%=====================================================================
\section{Examples}\label{sec:examples}
%=====================================================================

\subsection{\texorpdfstring{$\SU(p,q)$ on $\PP^{n-1}$}{SU(p,q) on P\textasciicircum (n-1)}}
There are three $G_0$-orbits: $D_\pm=\{\pm h>0\}$ and $\Sigma=\{h=0\}$. The base cycles are $\PP(E_\pm)$, and $M=\Sigma$ for the closed orbit, whose Levi form has signature $(p-1,q-1)$. Theorem~\ref{thmD} gives:
\begin{itemize}
\item $p,q\ge2$: only constants near any Matsuki orbit;
\item $p=1$: $\Sigma$ is the unit sphere in $\Hom(\C,\C^q)=\C^q$ and its hull is the closed ball;
\item $p=q=1$: $\Sigma$ is a circle, its own hull.
\end{itemize}
For $p,q\ge2$ the constancy is also seen directly: $\Sigma$ contains the projective lines $\PP(L)$ of totally isotropic $2$-planes $L$, and any two points of $\Sigma$ are joined by at most two of them (Lemma~\ref{lem:isolines}, Proposition~\ref{prop:curves}).

\begin{example}[Continuous CR functions versus germs]\label{ex:sqrt}
Let $p=1\le q-1$, so that $\Sigma=S^{2q-1}=\{|\Phi|=1\}\subset\C^q=\cA$ and $\widehat\Sigma=\overline\cD$ is the closed unit ball. By Theorem~\ref{thmF}(v), every continuous CR function on $\Sigma$ extends to a function on $\overline\cD$ that is holomorphic on $\cD$; this is the Bochner--Hartogs theorem. The extension is not, in general, holomorphic near any point of $\Sigma$. For instance $f(\Phi)=\sqrt{1-\Phi_1}$ (principal branch) is holomorphic on $\cD$ and continuous on $\overline\cD$, since $\operatorname{Re}(1-\Phi_1)\ge0$ there; it is a uniform limit of the functions $f(t\,\cdot)$, $t<1$, which are holomorphic near $\overline\cD$, so $f|_\Sigma\in CR^0(\Sigma)$. Suppose $F$ were holomorphic on a connected neighbourhood $U$ of $e=(1,0,\dots,0)$ with $F=f$ on $\Sigma\cap U$. On the line $\{(\zeta,0,\dots,0)\}$ the functions $F(\zeta,0,\dots,0)$ and $\sqrt{1-\zeta}$ are holomorphic in $\{|\zeta|<1\}\cap U$, continuous up to the arc $\{|\zeta|=1\}\cap U$ and equal there. By the Schwarz reflection principle they agree near $\zeta=1$, which is impossible because $\sqrt{1-\zeta}$ is not holomorphic at $1$. So continuous CR functions, unlike germs (Theorem~\ref{thm:hull}), need not extend to a neighbourhood of $\widehat M$.
\end{example}

\subsection{\texorpdfstring{$\SU(2,2)$ on $\Gr_2(\C^4)$}{SU(2,2) on Gr2(C4)}}
There are six $G_0$-orbits, listed in Table~\ref{tab:SU22}. Here $k=p=q=2$.
\begin{itemize}
\item $M(1,0,1)$ and $M(0,1,1)$ are of holomorphic type. In the chart, $M(1,0,1)$ is the set of rank-one partial isometries of $\C^2$, a $5$-dimensional strictly pseudoconvex hypersurface of the quadric cone $\{\det\Phi=0\}$. Its hull is $\{\det\Phi=0,\ \|\Phi\|\le1\}$, which contains the vertex $\Phi=0$, i.e.\ $W=\Ep$.
\item $M(0,0,2)\cong\U(2)$ is the exceptional case.
\item The base cycles $\{\Ep\}$ and $\{\Em\}$ are points.
\item The base cycle $\PP^1\times\PP^1$ of $O(1,1,0)$ is not of holomorphic type, so it has only constant holomorphic germs.
\end{itemize}

\begin{table}[ht]
\centering\small
\renewcommand{\arraystretch}{1.2}
\begin{tabular}{@{}lcccl@{}}
\toprule
$(a,b,c)$ & $\dim_\R\gamma$ & $\dim_\R\kappa$ & $\dim_\R M$ & holomorphic germs near $M$\\
\midrule
$(2,0,0)$ & 8 & 0 & 0 & point\\
$(0,2,0)$ & 8 & 0 & 0 & point\\
$(1,1,0)$ & 8 & 4 & 4 & constants\\
$(1,0,1)$ & 7 & 6 & 5 & hull $\{\det=0,\ \|\Phi\|\le1\}$\\
$(0,1,1)$ & 7 & 6 & 5 & hull $\{\det=0,\ \|\Psi\|\le1\}$\\
$(0,0,2)$ & 4 & 8 & 4 & totally real $\U(2)$, own hull\\
\bottomrule
\end{tabular}
\caption{$\SU(2,2)$ on $\Gr_2(\C^4)$.}\label{tab:SU22}
\end{table}

\subsection{\texorpdfstring{$\SU(2,6)$ on $\Gr_3(\C^8)$}{SU(2,6) on Gr3(C8)}}
Here $k=3\notin\{2,6\}$, so by Theorem~\ref{thm:vanishing} \emph{every} Matsuki orbit, including all base cycles, has only constant holomorphic germs in $Z$. The orbits are listed in Table~\ref{tab:SU26}.

\begin{table}[ht]
\centering\small
\renewcommand{\arraystretch}{1.2}
\begin{tabular}{@{}lllccl@{}}
\toprule
$(a,b,c)$ & base & fibre $\Iso_c$ & $\dim_\R M$ & CR dim/codim & Levi form in $\kappa$\\
\midrule
$(2,1,0)$ & $\mathrm{pt}\times\PP^5$ & pt & 10 & complex & --\\
$(1,2,0)$ & $\PP^1\times\Gr_2(\C^6)$ & pt & 18 & complex & --\\
$(0,3,0)$ & $\Gr_3(\C^6)$ & pt & 18 & complex & --\\
$(1,1,1)$ & $\PP^1\times\PP^5$ & $S^9$ & 21 & $10/1$ & $(5,1,4)$\\
$(0,2,1)$ & $\Gr_2(\C^6)$ & $\Iso_1(\C^{2,4})$ & 25 & $12/1$ & indefinite\\
$(0,1,2)$ & $\PP^5$ & $\Iso_2(\C^{2,5})$ & 26 & $11/4$ & --\\
\bottomrule
\end{tabular}
\caption{$\SU(2,6)$ on $\Gr_3(\C^8)$. The signature $(5,1,4)$ (positive, negative, zero) is given by Proposition~\ref{prop:levi}. For $M(0,2,1)$ the Levi form restricted to the fibre $\Iso_1(\C^{2,4})$, a hypersurface of Levi signature $(1,3)$ in $\PP^5$, is already indefinite.}\label{tab:SU26}
\end{table}

\begin{example}[The ball-bundle guess]\label{ex:ballbundle}
The orbit $M(1,1,1)$ is the unit sphere bundle of $L=\cO(-1)\boxtimes Q$ over $\PP^1\times\PP^5$. Each fibre sphere $S^9$ bounds the ball of positive lines, and the disc bundle $\{|A|<1\}$ lies in $O(2,1,0)$. Its zero section maps to the base cycle $\{\Ep\oplus\ell\}\cong\PP^5$, collapsing the $\PP^1$ factor. It is natural to guess that this disc bundle is the hull. Fibrewise Hartogs extension does show that every function holomorphic near $M$ in $\kappa$ extends to the disc bundle. But by Theorem~\ref{thmE} and Theorem~\ref{thm:vanishing} all such functions are constant, so the guess holds only vacuously. By Theorem~\ref{thmF}(iv) the same holds for continuous CR functions: although each fibre sphere carries many CR functions, none of them glue to a non-constant global one. The obstruction is cohomological: $H^0(\PP^5,S^dQ^*)=0$ for $d\ge1$. This is exactly the vanishing used in the proof of Theorem~\ref{thmE} for $\kappa(1,1)$, in the case $d=m_1<m_2$ with $s=1$: there $\Gr_s(\Em)=\PP^5$ and $Q_-^*=W_-^0$ is a proper subbundle of the trivial bundle $\Em^*$, so Lemma~\ref{lem:vanish} kills all non-zero $\lambda$. The Levi form has signature $(5,1)$ with kernel the tangent spaces of the $\PP^4$'s $\{\la e_1,e_2+e_3,\ell\ra:\ell\in\PP\la e_4,\dots,e_8\ra\}\subset M$ (Proposition~\ref{prop:levi}).
\end{example}

\begin{example}[Shilov boundary]
For $k=p<q$ and $c=p$, $M=O(0,0,p)$ is the closed $G_0$-orbit of $\Gr_p(\C^{p+q})$, the set of isometries $\Ep\to\Em$, i.e.\ the Shilov boundary of $\cD$. Then $V_p=\Hom(\Ep,\Em)$, and Theorem~\ref{thm:hull} recovers the fact that its hull is $\overline{\cD}$. Kaup and Zaitsev \cite{KZ00} computed this polynomial hull for Shilov boundaries of all bounded symmetric domains, together with continuous extension of CR functions. For instance, for $\SU(2,3)$ on $\Gr_2(\C^5)$, $M$ is a Stiefel manifold of dimension $8$ and CR codimension $4$.
\end{example}

%=====================================================================
\section{Concluding summary}\label{sec:summary}
%=====================================================================

For a Matsuki orbit $M=\kappa\cap\gamma$ of a real form in a flag manifold, the analytic behaviour of a neighbourhood of $M$ is decided by $K$-stable divisors. If a $K$-stable prime divisor misses $\kappa$, then $M$ lies in the preimage of an affine variety and carries infinitely many independent holomorphic germs (Theorem~\ref{thmC}(i)). This occurs for Hermitian and for non-Hermitian real forms (Example~\ref{ex:general}). If every such divisor contains $\kappa$, then, when stabilisers on the open $K$-orbit are connected, only constants occur (Theorem~\ref{thmC}(ii)).

For $\SU(p,q)$ on $\Gr_k(\C^{p+q})$ the $K$-stable prime divisors are $\{W\cap\Em\ne0\}$ when $k=p$ and $\{W\cap\Ep\ne0\}$ when $k=q$. The criterion therefore singles out the orbits of holomorphic type, and for them the hull is $\overline{\cD}\cap\{\rank\le c\}$ (Theorem~\ref{thmD}). The same dichotomy governs:
\begin{itemize}
\item regular functions on $K$-orbits (Theorem~\ref{thmE});
\item smooth and continuous global CR functions, and the degree-zero comparison with the Dolbeault cohomology of $\kappa$ (Theorem~\ref{thmF});
\item sections of all line bundles (Theorem~\ref{thmG}).
\end{itemize}
It is distinct from the Gindikin--Matsuki dichotomy for cycle spaces, and meets it only at the base point of $\cD$ (Theorem~\ref{thmH}).

%=====================================================================

\end{document}